\documentclass{amsart}
\usepackage{amssymb,latexsym,amsmath,amsfonts,hhline}
\usepackage{pdfsync}
\usepackage{ucs}
\usepackage{amssymb}
\usepackage{amsthm}
\usepackage{amsmath}
\usepackage{latexsym}
\usepackage[cp1251]{inputenc}
\usepackage[english]{babel}
\usepackage[mathcal]{eucal}
\usepackage{graphicx}
\usepackage{wrapfig}
\usepackage{caption}
\usepackage{subcaption}

\usepackage[left=3cm,right=3cm,top=3cm,bottom=3cm,bindingoffset=0cm]{geometry}

\makeatletter
\renewcommand{\subsection}{\@startsection{subsection}{2}{0mm}{-2mm}{-2mm}{\bf\normalsize}}

\makeatother

\newtheorem{formula}{}[section]
\newtheorem{definition}[formula]{Definition}
\newtheorem{corollary}[formula]{Corollary}
\newtheorem*{remark}{Remark}
\newtheorem{lemma}[formula]{Lemma}
\newtheorem{theorem}[formula]{Theorem}
\newtheorem{prop}{Proposition}[section]
\newtheorem{theo}{Theorem}[section]

\newtheorem{lemm}{Lemma}[section]

\newtheorem{corl}{Corollary}[section]

\theoremstyle{definition}
\newtheorem{defn}{Definition}[section]

\theoremstyle{remark}

\def\thrm{\begin{theorem}}
\def\thrml#1{\begin{theorem}\label{#1}}
\def\ethrm{\end{theorem}}
\def\rmrk{\begin{remark}}
\def\rmrkl#1{\begin{remark}\label{#1}}
\def\ermrk{\end{remark}}
\def\dfntn{\begin{definition}}
\def\dfntnl#1{\begin{definition}\label{#1}}
\def\edfntn{\end{definition}}
\def\nmrt{\begin{enumerate}}
\def\enmrt{\end{enumerate}}

\def\qtn{\begin{equation}}
\def\qtnl#1{\begin{equation}\label{#1}}
\def\eqtn{\end{equation}}
\def\lmm{\begin{lemma}}
\def\lmml#1{\begin{lemma}\label{#1}}
\def\elmm{\end{lemma}}
\def\crllr{\begin{corollary}}
\def\crllrl#1{\begin{corollary}\label{#1}}
\def\ecrllr{\end{corollary}}
\def\css{\begin{cases}}
\def\ecss{\end{cases}}

\DeclareMathOperator{\aut}{Aut}

\DeclareMathOperator{\sym}{Sym}
\DeclareMathOperator{\rad}{rad}

\begin{document}
\title{On  Schur $p$-groups of odd order}
\author{Grigory Ryabov}
\address{Novosibirsk State University, 2 Pirogova St., 630090 Novosibirsk, Russia}
\email{gric2ryabov@gmail.com}
\thanks{\rm The work is supported by the Russian Foundation for Basic Research (project 13-01-00505)}
\date{}

\begin{abstract}
 A finite group $G$ is called a Schur group if any $S$-ring over $G$ is associated in a natural way with a subgroup of $\sym(G)$ that contains all right translations. We prove that the groups $\mathbb{Z}_3\times \mathbb{Z}_{3^n}$, where $n\geq 1$, are  Schur. Modulo previously obtained results, it follows that every noncyclic Schur $p$-group, where $p$ is an odd prime, is isomorphic to $\mathbb{Z}_3\times \mathbb{Z}_3 \times \mathbb{Z}_3$ or $\mathbb{Z}_3\times \mathbb{Z}_{3^n}$, $n\geq 1$ .
\end{abstract}

\maketitle

\section{Introduction}
	
	Let $G$ be a finite group, $e$  the identity element of $G$. Let  $\mathbb{Z}G$ be the integer group ring. Given $X\subseteq G$,  denote the element $\sum_{x\in X} {x}$ of $\mathbb{Z}G$ by $\underline{X}$ and the set $\{x^{-1}:x\in X\}$ by $X^{-1}$.
	\begin{defn}
A subring  $\mathcal{A}$ of  $\mathbb{Z} G$ is called an \emph{$S$-ring} over $G$ if there exists a partition $\mathcal{S}=\mathcal{S}(\mathcal{A})$ of $G$ such that:

 $(1)$ $\{e\}\in\mathcal{S}$,

 $(2)$ $X\in\mathcal{S}\ \Rightarrow\ X^{-1}\in\mathcal{S}$,

 $(3)$ $\mathcal{A}=Span_{\mathbb{Z}}\{\underline{X}:\ X\in\mathcal{S}\}$.

The elements of this partition are called \emph{the basic sets}  of the $S$-ring $\mathcal{A}$.
\end{defn}

The definition of an $S$-ring goes back to Schur and Wielandt, they used ``the $S$-ring method'' to study a permutation group having a regular subgroup \cite{Schur,Wi}.

Let $\Gamma$ be a subgroup of $\sym(G)$ that contains  the subgroup of right translations $G_{right}=\{x\mapsto xg,~x\in G:g\in G\}$. Let $\Gamma_e$ stand for the stabilizer of $e$ in $\Gamma$ and  $Orb(\Gamma_e,G)$ stand for the set of all orbits of $\Gamma_e$ on $G$. As Schur proved in  \cite{Schur}, the $\mathbb{Z}$-submodule
$$\mathcal{A}=\mathcal{A}(\Gamma,~G)=Span_{\mathbb{Z}}\left\{\underline{X}:~X\in Orb(\Gamma_e,~G)\right\},$$
is an $S$-ring over $G$.

\begin{defn}[P\"{o}schel, 1974]
An $S$-ring $\mathcal{A}$ over  $G$ is called \emph{schurian} if $\mathcal{A}=\mathcal{A}\left(\Gamma,~G\right)$ for some $\Gamma$ with $G_{right}\leq \Gamma \leq \sym(G)$.
	\end{defn}
\begin{defn}[P\"{o}schel, 1974]
A finite group $G$ is called a \emph{Schur} group if every $S$-ring over $G$ is schurian.
\end{defn}

Wielandt wrote in  \cite{Wi2}:``Schur had conjectured for a long time that every $S$-ring is determined by a suitable permutation group'', or in our terms that every $S$-ring is schurian. However, it turns out to be false. The first counterexample  found by Wielandt \cite{Wi} is an $S$-ring over $\mathbb{Z}_p\times \mathbb{Z}_p$, where $p>3$ is a prime. The problem of determining all Schur groups was suggested by P\"{o}schel in \cite{Po} about 40 years ago. He proved that the cyclic $p$-groups are Schur whenever $p$ is odd. Moreover, if $p>3$ then a $p$-group is  Schur if and only if it is cyclic.  Applying this result P\"{o}schel and Klin solved the isomorphism problem for circulant graphs with $p^n$ vertices, where $p$ is an odd prime \cite{KP}.

Only 30 years later all cyclic Schur groups were classified  in  \cite{EKP1}. Strong necessary conditions of schurity for abelian groups were recently proved in \cite{EKP2}. From these conditions it follows that any abelian Schur group  belongs to one of  several explicitly given families. The case of non-abelian Schur groups was analyzed in \cite{PV}: it was proved   that every Schur group $G$ is solvable of derived length at most 2  and the number of distinct prime divisors of the order of $G$ does not exceed~$7$. 

In this article we are interested in abelian Schur $3$-groups. From  \cite[Theorem 1.2,~Theorem 1.3~] {EKP2} it follows that all abelian Schur $3$-groups are known except for the groups $\mathbb{Z}_3\times \mathbb{Z}_{3^n}$, where $n\geq 4$. We prove the following statement.

\begin{theo} \label{main}
The groups $\mathbb{Z}_3\times \mathbb{Z}_{3^n}$, where $n\geq 1$, are Schur.
\end {theo} 
	
	Modulo previously obtained results \cite[Theorem 1.2,~Theorem 1.3]{EKP2}, \cite[Theorem 4.2]{PV}, \cite[Theorem~1]{Ry},  Theorem \ref{main} implies the classification of Schur $p$-groups, where  $p$ is an odd prime:
	
	\begin{theo}
	A $p$-group $G$, where  $p$ is an odd prime, is Schur if and only if $G$ is cyclic or $p=3$ and $G$ is isomorphic to one of the groups below:
	
	$(1)$ $\mathbb{Z}_3\times \mathbb{Z}_3 \times \mathbb{Z}_3;$
	
	$(2)$ $\mathbb{Z}_3\times \mathbb{Z}_{3^n}$, $n\geq 1.$  
	
	\end{theo}
	
The case of Schur $2$-groups was considered in \cite{MP3}, where the complete classification of abelian Schur $2$-groups was given. From \cite[Theorem 1.2]{MP3} it follows that 
any non-abelian Schur $2$-group of order at least $32$ is dihedral.

In the proof of Theorem \ref{main} we follow, in general, the scheme of the proof \cite[Theorem 10.1]{MP3}. At first, we establish that  regular and rational $S$-rings over $\mathbb{Z}_3\times \mathbb{Z}_{3^n}$ with trivial radical are schurian. Further, we prove that an $S$-ring over $\mathbb{Z}_3\times \mathbb{Z}_{3^n}$  with nontrivial radical  is a ``good'' generalized wreath product of two smaller $S$-rings. In this case the $S$-ring is schurian by the criterion of schurity for a generalized wreath product over abelian groups \cite[Theorem 10.2]{MP3}.
	
\section{Preliminaries}
	
	In this section we recall some definitions and facts concerned with $S$-rings and Schur groups. The most part of them is taken from \cite{MP3}. There are references before other statements.

	\begin{defn}

Let $G$ be a finite group, $\mathcal{R}$  a family of binary relations on $G$.
The pair \emph{$\mathcal{C}=\left(G,\mathcal{R}\right)$} is called a \emph{Cayley scheme} over $G$ if the following properties are satisfied:

$(1)$  $\mathcal{R}$ forms a partition of the set $G\times G$;

$(2)$  $Diag\left(G\times G\right)\in\mathcal{R}$;

$(3)$  $\mathcal{R}=\mathcal{R}^*$, i.\,e., if $R\in\mathcal{R}$ then $R^*=\{(h,g)\mid (g,h)\in R\}\in\mathcal{R}$;

$(4)$  if $R,~S,~T\in\mathcal{R}$, then the number $|\{h\in G:(f,h)\in R,~(h,g)\in S\}|$ does not depend on the choice of $(f,g)\in T$.

\end{defn}
	
	Let $\mathcal{A}$ be an  $S$-ring over $G$. We associate each basic set  $X\in \mathcal{S}(\mathcal{A})$  with the binary relation  $\{(a,xa) \mid a\in G, x\in X\}\subseteq G\times G$ and denote it by $R(X)$. The set of all such binary relations forms a partition $\mathcal{R}(\mathcal{S})$ of  $G\times G$.
	\begin{lemma}
$\mathcal{C}(\mathcal{A})=(G,\mathcal{R}(\mathcal{S}))$ is a Cayley scheme over $G$.
\end{lemma}
	
	The relation $R(X)$ is called \emph{the basic relation} of the scheme $\mathcal{C}(\mathcal{A})$ corresponding to $X$.
	
	\begin{defn}
Cayley schemes  $\mathcal{C}=(G,\mathcal{R})$ and $\mathcal{C}^{'}=(G^{'},\mathcal{R}^{'})$ are called \emph{isomorphic} if there exists a bijection $f:G\rightarrow G^{'}$ such that  $\mathcal{R}^{'}=\mathcal{R}^f$, where $\mathcal{R}^f=\{R^f:~R\in\mathcal{R}\}$ and $R^f=\{(a^f,~b^f):~(a,~b)\in R\}$.
	\end{defn}
We say that $S$-rings $\mathcal{A}$ over $G$  and $\mathcal{A}^{'}$ over $G^{'}$ are \emph{isomorphic} if the Cayley schemes associated to $\mathcal{A}$  and $\mathcal{A}^{'}$ are isomorphic. Any isomorphism between these schemes is called the isomorphism from $\mathcal{A}$  onto $\mathcal{A}^{'}$. The
group  of all isomorphisms from $\mathcal{A}$ to itself has a normal subgroup $\aut(\mathcal{A})$
that consists of all isomorphisms $f$ such that $X^f = X$ for all $X\in \mathcal{S}(\mathcal{A})$; any such $f$ is called an  \emph{automorphism} of the $S$-ring $\mathcal{A}$. An $S$-ring $\mathcal{A}$ over $G$ is schurian if and only if   $\mathcal{S}(\mathcal{A})=Orb(\operatorname {\aut}(\mathcal{A})_e,G)$. Under \emph{a Cayley isomorphism} from $\mathcal{A}$ to $\mathcal{A}^{'}$ we mean  a group isomorphism $f:G\rightarrow G^{'}$ such that $\mathcal{S}(\mathcal{A})^f=\mathcal{S}(\mathcal{A^{'}})$. In this case $\mathcal{A}$   and $\mathcal{A}^{'}$  are called \emph{Cayley isomorphic}. Every Cayley isomorphism  is an isomorphism.

\begin{defn}
Let $\mathcal{A}$ be an $S$-ring over $G$. A subset $X \subseteq G$ is called \emph{an $\mathcal{A}$-set} if $\underline{X}\in \mathcal{A}$.
\end{defn}

\begin{defn}
Let $\mathcal{A}$ be an $S$-ring over $G$. A subgroup $H \leq G$ is called \emph{an $\mathcal{A}$-subgroup} if $\underline{H}\in \mathcal{A}$.
\end{defn}

\begin{lemm}\label{restrict}
Let $\mathcal{A}$ be an $S$-ring over $G$ and $H$ be an $\mathcal{A}$-subgroup. Then the module

$$\mathcal{A}_H=Span_{\mathbb{Z}}\left\{\underline{X}:~X\in\mathcal{S}(\mathcal{A}),~X\subseteq H\right\},$$

is an $S$-ring over $H$.
In addition, if $\mathcal{A}$ is schurian, then $\mathcal{A}_H$ is schurian too.

\end{lemm}

\begin{defn}
Let $L,U$ be subgroups of a group $G$ and $L$ be normal in $U$. A section $U/L$ of $G$ is called an \emph{$\mathcal{A}$-section} if $U$ and $L$ are $\mathcal{A}$-subgroups.
\end{defn}

\begin{lemm}\label{ssection}
Let $F=U/L$ be an $\mathcal{A}$-section. Then the module
$$\mathcal{A}_F=Span_{\mathbb{Z}}\left\{\underline{X}^{\pi}:~X\in\mathcal{S}(\mathcal{A}),~X\subseteq U\right\},$$
where $\pi:U\rightarrow U/L$ is the quotient epimorphism, is an $S$-ring over $F$.

In addition, if $\mathcal{A}$ is schurian, then $\mathcal{A}_F$ is schurian too.
\end{lemm}
	
The following two lemmas give well-known properties of basic sets of an $S$-ring.

\begin{lemm}\label{eq}
For any basic sets $X,~Y,~Z$ of an $S$-ring $\mathcal{A}$ over $G$ denote by $c^Z_{X,Y}$ the number of distinct representations of an element $z \in Z$ in the form $z = xy$ with $x \in X$ and $y \in Y$. Then 
$$|Z|c^{Z^{-1}}_{X,Y}=|X|c^{X^{-1}}_{Y,Z}=|Y|c^{Y^{-1}}_{Z,X}.$$

\end{lemm}

\begin{lemm}\label{aset}
For any basic sets $X,~Y$ of an $S$-ring $\mathcal{A}$ over $G$ the set $XY$ is an $\mathcal{A}$-set. Moreover, if $|X|=1$ or $|Y|=1$ then $XY$ is a basic set.

\end{lemm}

\begin{lemm}  \label{intersection}
Let $H$ be an $\mathcal{A}$-subgroup of $G$ and $X \in \mathcal{S}(\mathcal{A})$. Then the number $|X\cap Hg|$ does not depend on $g \in G$ with $X \cap Hg \neq \varnothing$.

\end{lemm}

\begin{defn}
Let $G$ be a finite group and $X \subseteq G$. Then the subgroup of $G$ defined by $\rad(X)=\{g\in G| Xg=gX=X\}$ is called \emph{the radical} of $X$.
\end{defn}

\begin{lemm} \label{radbasic}
Let $\mathcal{A}$ be an $S$-ring  over $G$. If a set $X$ is  an $\mathcal{A}$-set,  then the groups $\langle X \rangle$ and $\rad(X)$ are  $\mathcal{A}$-subgroups of $G$.
\end{lemm}

\begin{theo}  \label{separat}
Let $X$ be a basic set of an $S$-ring $\mathcal{A}$ over  $G$. Suppose that $H\leq \rad(X \setminus H)$ for some group $H$ such that $X\cap H \neq \varnothing$ and $X \setminus H \neq \varnothing$. Then $X=\langle X \rangle \setminus \rad(X)$ with $\rad(X) \leq H \cap \langle X \rangle$. In this case $H$ is called a separating group.
\end{theo}

\begin{defn}
An $S$-ring $\mathcal{A}$ over $G$ is called \emph{primitive} if $G$ has no nontrivial proper $\mathcal{A}$-subgroups. Otherwise $\mathcal{A}$ is called imprimitive.
\end{defn}

	\begin{lemm}  \label{bgroups}
	For every prime $p$ a primitive $S$-ring over  $\mathbb{Z}_{p^a}\times \mathbb{Z}_{p^b},~a>b\geq 0,~a>1,$ has rank $2$.
	\end{lemm}
	
	\begin{proof}
	Follows from  the proofs of \cite[Theorem 25.3,~Theorem 25.5]{Wi}
	\end{proof}

	\begin{defn}
	Let $K \leq \aut(G)$. Then the orbit partition $Orb(K, G)$ defines an $S$-ring $\mathcal{A}$ over $G$.  In this case $\mathcal{A}$ is called \emph{cyclotomic} and denoted by $Cyc(K,G)$.
	\end{defn}
	
	For any set $X \subseteq G$ and $m\in \mathbb{Z}$ denote by $X^{(m)}$ the set $\{x^m: x \in X\}$. We say that two sets $X$ and $Y$ are \emph{rationally conjugate} if there exists $m\in \mathbb{Z}$ coprime to $n$ such that $Y=X^{(m)}$.

Further we formulate two theorems on multipliers of $S$-rings over abelian groups. Both of them were proved by Schur in \cite{Schur}.

\begin{theo} \label{burn}
Let $\mathcal{A}$ be an $S$-ring over an abelian group $G$ of order $n$. Then $X^{(m)}\in \mathcal{S}(\mathcal{A})$  for all $X\in \mathcal{S}(\mathcal{A})$ and all $m\in \mathbb{Z}$ coprime to $n$.
\end{theo}

	\begin{theo} \label{sch}
	Let $\mathcal{A}$ be an $S$-ring over an abelian group $G$. Then given a prime $p$ dividing the order of $G$, the set
	
$$X^{[p]}=\{x^p:x\in X,~|X\cap Hx|\not\equiv~0~\mod~p\},$$
	where $H=\{g\in G:g^p=e\}$, is an $\mathcal{A}$-set for all $\mathcal{A}$-sets $X$.
	
	\end{theo}
	
	The following lemma is taken from \cite{Po}.
	
	\begin{lemm} \label{cyclering}
	Let $\mathcal{A}$ be an $S$-ring over cyclic $p$-group $G$, where $p$ is an odd prime. Then for every $X\in \mathcal{S}(\mathcal{A})$ one of the following statements holds:
	
	$(1)$ $X\in Orb(K,G)$ for some $K\leq \aut(G)$;
	
	$(2)$ $X=\langle X \rangle \setminus \rad(X)$.
	
	\end{lemm}
	
	\begin{lemm}\label{cyclesection}
Let $\mathcal{A}$ be a cyclotomic $S$-ring over cyclic $3$-group $G$ and $S$ be an $\mathcal{A}$-section. Suppose that $\rad(\mathcal{A})=e$. Then $|\rad(\mathcal{A}_S)|=1.$

\end{lemm}

	\begin{proof}
	Follows from \cite[ Theorem 7.3]{EKP1}.
	\end{proof}

 Let $D=S\times C$, $S=\langle s\rangle,~|s|=3,~C=\langle c\rangle,~|c|=3^n,~n\geq 2$. Every subset $X$ of  $D$  can be  uniquely presented as the union $X=X_{0e}\cup sX_{1s}\cup s^2X_{2s^2}$, where $X_{0e},~X_{1s},~X_{2s^2} \subseteq C$. If $T\subseteq D$, we denote the set $\{t\in T: |t|\leq 3^m\}$  by $T_m$. Denote  one of two elements of order $3$ from $C$ by $c_1$, the cyclic group $\langle c_1 \rangle$ by $C_1$, the elementary abelian group $S\times C_1$ by $E$, the basic set of the $S$-ring $\mathcal{A}$ containing an element $x$ by $T_{x}$.

\begin{lemm} \label{oarg}
If a basic set $T$ of an $S$-ring $\mathcal{A}$ over  $D$ contains  elements $a,~b$ such that $|a|>|b|\geq 3$ then $a\{1,~c_1,~c_1^2\}\subseteq T$
\end{lemm}

\begin{proof}
Set  $m=1+\frac{|a|}{3},~l=1+\frac{2|a|}{3}$. Then  $b^m=b^l=b$. By Theorem \ref{burn}, the sets $T^{(m)},~T^{(l)}$ are basic and $T^{(m)}=T^{(l)}=T$. So $\{a^m,~a^l\}=\{aa^{\frac{|a|}{3}},~aa^{\frac{2|a|}{3}}\}=\{ac_1,~ac_1^2\}\subseteq T$. 
\end{proof}

Further we formulate and prove two lemmas that describe regular sets with trivial radical over $C=\mathbb{Z}_{3^n}$ and $D=\mathbb{Z}_3\times \mathbb{Z}_{3^n}$.

\begin{lemm} \label{cycleorbit}
Let $X\in Orb(M,C)$ where $M\leq \aut(C)$. Then $\rad(X)=e$ if and only if $M=e$ or $M=\{e,~\delta\}$, where $\delta:x\rightarrow x^{-1}$. This means that  there exists  $x \in C$ such that $X=\{x\}$ or $X=\{x,~x^{-1}\}$.
\end{lemm}

\begin{proof}
Follows from  \cite[Lemma 5.1]{EKP3}.
\end{proof}

\begin{lemm}\label{lowset}
Let $T$ be a basic set of an $S$-ring $\mathcal{A}$ over $D$ and $3^m=min \{|t|:t\in T\}$. Suppose that $\rad(T)=e$. Then there exist subsets $X,Y,Y_1,Z,Z_1$ of $C$ such that $T_m=X\cup sY\cup s^2Y_1\cup sZ\cup s^2Z_1,~Y\cap Y_1=Z\cap Z_1=\varnothing$, and  one of the following statements holds:

\begin{enumerate}
\item every nonempty set $U\in\{X,Y\cup Y_1,Z\cup Z_1\}$ is singleton;

\item every nonempty set $U\in\{X,Y\cup Y_1,Z\cup Z_1\}$ is of the form $\{x,x^{-1}\},~x\in C$.

\end{enumerate}

\end{lemm}

\begin{proof}
Let us present $T_m$  as a union $T_m=T_{0e,m}\cup sT_{1s,m}\cup s^2T_{2s^2,m}$, where $T_{0e,m},~T_{1s,m},~T_{2s^2,m}\subseteq C$. Denote by $K$ the group $\{\sigma_m:m\in \mathbb{Z}^{*}_{3^n}\},~\sigma_m: x\rightarrow x^m,$ and by $M$ the setwise stabilizer $K_{\{T\}}$ of $T$ in $K$. If $a\in T_{0e,m}$ and $\alpha \in M$ then $a\alpha \in T_{0e,m}$. Thus $T_{0e,m}$ is a union of some orbits of $M$. Since all elements in $T_{0e,m}$ have the same order, $T_{0e,m}$ is an orbit of $M$. If $a\in sT_{1s,m} \cup s^2T_{2s^2,m}$ and $\alpha \in M$ then $a\alpha \in sT_{1s,m} \cup s^2T_{2s^2,m}$. Thus $sT_{1s,m} \cup s^2T_{2s^2,m}$ is a union of some orbits of $M$. Moreover, it is easy to check that $sT_{1s,m} \cup s^2T_{2s^2,m}$ is a union of at most two orbits  of $M$. Therefore $T_m$ can be presented as a union $T_m=X\cup sY\cup s^2Y_1\cup sZ\cup s^2Z_1$, where each set $X,~Y\cup Y_1,~Z\cup Z_1$ is empty or  equal to an orbit of $M$. At least one of the sets $X,~Y\cup Y_1,~Z\cup Z_1$ has trivial radical because otherwise $c_1\in \rad(T)$ by Lemma \ref{oarg}. Without loss of generality we assume that $X\neq \varnothing$ and $\rad(X)=e$. Then from Lemma \ref{cycleorbit} it follows that $M=e$ or $M=\{e,~\delta\}$ where $\delta:x\rightarrow x^{-1}$. Now the claim follows. 
\end{proof}
	
Now we describe two constructions producing $S$-rings over the group $G=G_1\times G_2$ from $S$-rings $\mathcal{A}_1$ and $\mathcal{A}_2$ over $G_1$ and $G_2$ respectively.
	
	\begin{defn} 
	Let $\mathcal{A}_1$ be an $S$-ring over $G_1$ and $\mathcal{A}_2$ be an $S$-ring over $G_2$. Then the set
	
	$$\mathcal{S}=\mathcal{S}(\mathcal{A}_1)\times \mathcal{S}(\mathcal{A}_2)=\{X_1\times X_2:~X_1\in \mathcal{S}(\mathcal{A}_1),~X_2\in \mathcal{S}(\mathcal{A}_2)\} $$
forms a partition of the group $G=G_1\times G_2$ that defines an $S$-ring  over $G$. This $S$-ring is called \emph{the tensor product} of the $S$-rings $\mathcal{A}_1$ and $\mathcal{A}_2$ and 	denoted by $\mathcal{A}_1 \otimes \mathcal{A}_2$.
	\end{defn}

	\begin{defn} 
	Let $\mathcal{A}_1$ be an $S$-ring over group $G_1$ with the identity element $e_1$ and $\mathcal{A}_2$ be an $S$-ring over group $G_2$ with the identity element $e_2$. Then the set $\mathcal{S}=\mathcal{S}_1 \cup \mathcal{S}_2$ where
	
	$$\mathcal{S}_1=\{X_1\times \{e_2\}:~X_1\in \mathcal{S}(\mathcal{A}_1)\},~\mathcal{S}_2=\{G_1\times \{X_2\}:~X_2\in \mathcal{S}(\mathcal{A}_2)\setminus \{e_2\}\} $$
forms a partition of the group $G=G_1\times G_2$ that defines an $S$-ring  over $G$. This $S$-ring is called \emph{the wreath product} of the $S$-rings $\mathcal{A}_1$ and $\mathcal{A}_2$ and 	denoted by $\mathcal{A}_1 \wr \mathcal{A}_2$.
	\end{defn}

	\begin{lemm} \label{schurtens}
	The $S$-ring $\mathcal{A}_1 \otimes \mathcal{A}_2$ is schurian if and only if so are the  $S$-rings $\mathcal{A}_1$ and $\mathcal{A}_2$.
	\end{lemm}
	
	\begin{lemm} \label{schurwr}
	The $S$-ring $\mathcal{A}_1 \wr \mathcal{A}_2$ is schurian if and only if so are the  $S$-rings $\mathcal{A}_1$ and $\mathcal{A}_2$.
	\end{lemm}

	\begin{defn}
	Let $\mathcal{A}$ be an $S$-ring over $G$ and $L,~U$ be $\mathcal{A}$-subgroups such that $e \leq L \leq U \leq G,~L \unlhd G$. The $S$-ring $\mathcal{A}$ is called \emph{the generalized wreath product} or \emph{$U/L$-wreath product} of $\mathcal{A}_U$ and $\mathcal{A}_{G/L}$ if every basic set $S\in \mathcal{S}(\mathcal{A}) $ outside  $U$ is a union of $L$-cosets. The product is called \emph{nontrivial} or \emph{proper} if $L\neq 1,~ U\neq G$. 
\end{defn}

\begin{remark}
When $U=L$ the generalized wreath product coincides with a wreath product.
\end{remark}

\begin{defn}
Permutation groups $\Gamma,~\Gamma_1\leq \sym(V)$   are called  \emph{$2$-equivalent} if  $Orb(\Gamma,V^2)=Orb(\Gamma_1,V^2)$.

\end{defn}

\begin{defn}
A permutation group  $\Gamma\leq \sym(V)$ is called \emph{$2$-isolated} if  it is the only group which is $2$-equivalent to $\Gamma$.

\end{defn}

\begin{theo} \label{genwr}
Let $\mathcal{A}$ be an $S$-ring over an abelian group $G$. Suppose that $\mathcal{A}$ is an $U/L$-wreath product and $S$-rings $\mathcal{A}_U$ and $\mathcal{A}_{G/L}$ are schurian. Then $\mathcal{A}$ is schurian if and only if there exist groups $\Delta_0\geq (G/L)_{right}$ and $\Delta_1\geq U_{right}$ such that $\Delta_0$ is $2$-equivalent to $\aut(\mathcal{A}_{G/L})$, $\Delta_1$ is $2$-equivalent to $\aut(\mathcal{A}_U)$ and $(\Delta_0)^{U/L}=(\Delta_1)^{U/L}$.

\end{theo}	
	
\begin{corl}\label{isolated}
Under the hypothesis of Theorem \ref{genwr} the $S$-ring $\mathcal{A}$ is schurian whenever the group $\aut(\mathcal{A}_{U/L})$ is $2$-isolated.

\end{corl}

\begin{lemm}\label{regorb}
Let $\mathcal{A}$ be an $S$-ring over $G$. Suppose that the point stabilizer of $\aut(\mathcal{A})$ has a faithful regular orbit. Then $\aut(\mathcal{A})$  is $2$-isolated.
\end{lemm}

\begin{defn}
An $S$-ring $\mathcal{A}$ is called \emph{quasi-thin} if $|X|\leq 2$ for every $X\in \mathcal{S}(\mathcal{A})$.
\end{defn}

\begin{defn}
A basic set $X\neq \{e\}$ of a quasi-thin $S$-ring $\mathcal{A}$ is called \emph{an orthogonal}, if $X \subseteq Y Y^{-1}$ for some $Y\in \mathcal{S}(\mathcal{A})$.
\end{defn}

\begin{lemm}\label{quasithin}
Any commutative quasi-thin $S$-ring $\mathcal{A}$ is schurian. Moreover, if it has at least two orthogonals, then the group $\aut(\mathcal{A})_e$ has a faithful regular orbit.
\end{lemm}

\section{$S$-rings over $D=\mathbb{Z}_3\times \mathbb{Z}_{3^n}$: basic sets containing elements of order $3$ }

A set $X \subset D$ is called \emph{highest} (in $D$) if it contains an element of order $3^n$. Given an $S$-ring $\mathcal{A}$ over  $D$ set $\rad(\mathcal{A})$ to be the group generated by the groups $\rad(X)$, where $X$ runs over the highest basic sets of $\mathcal{A}$. Clearly, $\rad(\mathcal{A}) = e$ if and only if every highest basic set of $\mathcal{A}$  has trivial radical. A set $X \subset D$ is called \emph{regular} if it consists of  elements of the same order. An $S$-ring $\mathcal{A}$ over a group $D$ is called \emph{regular} if each highest basic set of $\mathcal{A}$ is regular. A set $X \subset D$ is called \emph{rational} if $X=\cup_m X^{(m)}$, where $m$ runs over integers coprime to $3$. An $S$-ring $\mathcal{A}$ over a group $D$ is called \emph{rational} if each  basic set of $\mathcal{A}$ is rational. 

In this section we describe the basic sets containig an element of order $3$. The main result of this section is given by the following two lemmas.

\begin{lemm}\label{Tc1}
Let $\mathcal{A}$ be an $S$-ring over  $D$. Then exactly one of the following statements holds:

$(1)$ $T_{c_1}$ is rational, nonregular, and $T_{c_1}\cup L$ is an $\mathcal{A}$-subgroup for some $\mathcal{A}$-subgroup $L\leq E$ such that $L\cap C_1=\{e\}$;

$(2)$ $T_{c_1}$ is regular. In this case $C_1$ is an $\mathcal{A}$-subgroup or $E$ is an $\mathcal{A}$-subgroup.

\end{lemm}

\begin{lemm}\label{Ts}
Let $\mathcal{A}$ be an $S$-ring over  $D$ and $T_{c_1}$ is rational and nonregular. Denote the set of all basic sets that contain an element of order $3$ by $I$. Then exactly one of the following statements holds:

$(1)$ $I=\{T_{c_1},\{q,q^2\},\{q,q^2\}T_{c_1}\},$ where $q\in E\setminus C_1$ and $T_{c_1}\cup \{e\}$ is a cyclic $\mathcal{A}$-subgroup;

$(2)$ $I=\{T_{c_1},\{q,q^2\}\},$ where $q\in E\setminus C_1$ and $T_{c_1}\cup \{e,q,q^2\}$ is an $\mathcal{A}$-subgroup;

$(3)$ $I=\{T_{c_1}\}$ ;

$(4)$ $I=\{T_{c_1},\{q\},\{q^2\},qT_{c_1},q^2T_{c_1}\},$ where $q\in E\setminus C_1$ and $T_{c_1}\cup \{e\}$ is a cyclic $\mathcal{A}$-subgroup;

$(5)$ $C_1\leq \rad(T)$ for every $T\in I\setminus \{T_{c_1}\}$.

\end{lemm}

The proofs of Lemma \ref{Tc1} and Lemma \ref{Ts} will be given later. Now we need to prove an auxiliary lemma.

\begin{lemm}\label{regTc1}
Let $T_{c_1}$ be nonregular. Then $T_{c_1}$ is rational.
\end{lemm}

\begin{proof}
Assume the contrary. Let $T_{c_1}\cap E=Y$, $|T_{c_1}|=a$, and $\alpha=|Y|-|c_1Y\cap Y|$. Note that $\alpha \geq 1$ because $c_1\in Y$ but $c_1\notin c_1Y$. Since $T_{c_1}$ is not rational, $|Y|\leq 4$ and if $|Y|=4$ then  $|c_1Y\cap Y|\geq 1$. Therefore $\alpha \leq 3$. Lemma \ref{oarg} implies that $C_1\leq \rad(T_{c_1})\setminus Y$. So  

$$C^{T_{c_1}}_{T_{c_1}T_{c_1}^{-1}}=|c_1T_{c_1}\cap T_{c_1}|=a-\alpha.$$ 
Hence every element from $T_{c_1}\cup T_{c_1^2}$ enters the element $\underline{T_{c_1}}~\underline{T_{c_1}}^{-1}$ with coefficient $a-\alpha$. Since $e$ enters $\underline{T_{c_1}}~\underline{T_{c_1}}^{-1}$ with coefficient $a$, we conclude that  at least $2a^2-2\alpha a +a$ elements enter $\underline{T_{c_1}}~\underline{T_{c_1}}^{-1}$. On the other hand, exactly $a^2$ elements enter this element. Thus we arrive to a contradiction if $2a^2-2\alpha a +a>a^2$ or equivalently $a>2\alpha-1$. However, this holds for $\alpha=1$ because $T_{c_1}$ is nonregular; if $\alpha=2$  then $a>4$ by Lemma \ref{oarg}; if $\alpha=3$ then $a>5$ by Lemma \ref{oarg}. 
\end{proof}

\begin{proof}[Proof of  Lemma \ref{Tc1}]

If $T_{c_1}$ is not rational then $T_{c_1}$ is regular by Lemma \ref{regTc1}. In this case $C_1=\langle T_{c_1} \rangle$ or $E=\langle T_{c_1} \rangle$ and Lemma \ref{radbasic} implies Statement $(2)$ of the lemma. If $T_{c_1}$ is regular and rational, obviously, Statement $(2)$ also holds. Suppose that $T_{c_1}$ is rational and nonregular. Then it follows from Lemma \ref{oarg} that $c_1,~c_1^2 \in \rad(T_{c_1}\setminus E )$. The set  $T_{c_1}\cap E$ is a union of $1,~2,~3$ or $4$ subgroups of order $3$ without $\{e\}$. So there are four possibilities for $T_{c_1}\cap E$:

\begin{enumerate}
\item $T_{c_1}\cap E=\left\{c_1,~c_1^2\right\};$
\item $T_{c_1}\cap E=\left\{c_1,~c_1^2,~q,~q^2\right\} $ where $|q|=3;$
\item $T_{c_1}\cap E=\left\{c_1,~c_1^2,~q,~q^2,~f,~f^2\right\} $ where $|q|=|f|=3;$
\item $T_{c_1}\cap E=\left\{c_1,~c_1^2,~s,~s^2,~sc_1,~s^2c_1,~sc_1^2,~s^2c_1^2\right\}.$
\end{enumerate}

If $T_{c_1}=\langle T_{c_1}\rangle \setminus \rad(T_{c_1})$ and $\rad(T_{c_1})=\{e\}$ or $|\rad(T_{c_1})|=3$ then Statement $(1)$ of the lemma holds with $L=\rad(T_{c_1})$. Let us prove that in all cases $T_{c_1}=\langle T_{c_1}\rangle \setminus \rad(T_{c_1})$ and $\rad(T_{c_1})=\{e\}$ or $|\rad(T_{c_1})|=3$. Note that $T_{c_1}\cap C_1\neq \varnothing$ and $T_{c_1}\setminus C_1 \neq \varnothing$.  In the first and fourth cases Lemma~\ref{oarg} implies that $C_1\leq \rad(T\setminus C_1)$. By Theorem \ref{separat}, we have $T_{c_1}=\langle T_{c_1}\rangle \setminus \rad(T_{c_1})$. In the first case $\rad(T_{c_1})$ does not contain $c_1$. If $\rad(T_{c_1})$ contains an element of order $m>3$ then $x^{\frac{m}{3}}=c_1\in \rad(T_{c_1})$ or $x^{\frac{2m}{3}}=c_1\in \rad(T_{c_1})$.  Therefore $\rad(T_{c_1})=\{e\}$ or $|\rad(T_{c_1})|=3$. In the fourth case $\rad(T_{c_1})=\{e\}$. In the second case 

$$|c_1T_{c_1}\cap T_{c_1}|=|T_{c_1}|-3=C^{T_{c_1}}_{T_{c_1}T_{c_1}}=|qT_{c_1}\cap T_{c_1}|. \eqno(1)$$
Note that  $q\{c_1,~c_1^2,~q,~q^2\}\cap T_{c_1}=\{q^2\}$. Therefore $|q(T_{c_1}\setminus E)\cap T_{c_1}|=|T_{c_1}|-4=|T_{c_1}\setminus E|$. We  conclude that $q,~q^2\in \rad(T_{c_1}\setminus E)$. Also $c_1,~c_1^2\in \rad(T_{c_1}\setminus E)$  by Lemma \ref{oarg} and $T_{c_1}\cap E\neq \varnothing,~T_{c_1}\setminus E \neq \varnothing$. Thus, by Theorem \ref{separat} for separating group $E$, we have  $T_{c_1}= \langle T_{c_1} \rangle \setminus \rad(T_{c_1})$. Since   $\rad(T_{c_1})$ does not contain $c_1$, the group $\rad(T_{c_1})$ is trivial or has order $3$. In the third case $(1)$ holds and $|q\{c_1,~c_1^2,~q,~q^2~f,~f^2\}\cap T_{c_1}|=3$ because $\{f,f^2\}=\{qc_1,q^2c_1^2\}$ or $\{f,f^2\}=\{q^2c_1,qc_1^2\}$. So $|q(T\setminus E)\cap T|=|T|-6=|T\setminus E|$ and we have $q,~q^2\in \rad(T\setminus E )$. By Theorem \ref{separat} for separating group $E$, we have  $T_{c_1}= \langle T_{c_1}\rangle \setminus \rad(T_{c_1})$. Since   $\rad(T_{c_1})$ does not contain $c_1$, it is trivial or has order $3$.
\end{proof}

\begin {lemm}\label{sset}
In the conditions of Lemma \ref{Ts} let all basic sets from $I\setminus \{T_{c_1}\}$ are rational with trivial radical. Then one of Statements $(1)-(3)$ of Lemma~\ref{Ts} holds.
\end{lemm}

\begin{proof}
Statement $(2)$ of Lemma~\ref{Ts} obviously holds if $T_{c_1}\cup \{e,~q,~q^2\}\leq D$, where $q,~q^2 \in E\setminus C_1$. Statement $(3)$ holds if $T_{c_1}\cup \{e\}\leq D$ and  $T_{c_1}\cup \{e\}$ is noncyclic. Put $K=T_{c_1}\cup \{e\}$. By Lemma \ref{Tc1}, we may assume that  $K$ is a cyclic group and $K\leq C$.

If $X\in I\setminus \{T_{c_1}\}$ then $|X\cap E|\in \{2,4,6\}$. If $|X\cap E|=6$ then $C_1\leq \rad(T)$, a contradiction with the assumption of the lemma. So $I\setminus \{T_{c_1}\}=\{X,Y,Z\}$, where $|X\cap E|=|Y\cap E|=|Z\cap E|=2$, or $I\setminus \{T_{c_1}\}=\{X,Y\}$, where $|X\cap E|=2,~|Y\cap E|=4$. In both cases there exists $X\in I\setminus \{T_{c_1}\}$ such that $X\cap E=\{q,q^2\}$. Note that $|c_1X\cap X|=|X|-2$. It is obvious if $X$ is regular and follows from Lemma~\ref{oarg} otherwise. Therefore Lemma \ref{eq} implies that 

$$(|X|-2)|T_{c_1}|=C^{T_{c_1}}_{XX}|T_{c_1}|=C^{X}_{T_{c_1}X}|X|. \eqno(2)$$ 
Since $(|X|,|X|-2)\leq 2$, we conclude that $|X|=2$ or $|T_{c_1}|=\frac{l}{2}|X|$, where $l\geq 1$ is an integer. Moreover, $l\neq 1$ because $|T_{c_1}|\equiv |X|\equiv 2\mod 3$. Every element from $T_{c_1}$ enters the element $\underline{X}^2$ with coefficient $|X|-2$ because $|c_1X\cap X|=|X|-2$; every element from $X$ enters $\underline{X}^2$ because $|qX\cap X|\geq 1$; the identity element $e$ enters $\underline{X}^2$ with coefficient $|X|$ because $X=X^{-1}$. Assuming that $|X|>2$ and $|T_{c_1}|>\frac{3}{2}|X|$,  we conclude that at least $|X|^2-\frac{3}{2}|X|$ elements of $D$ (counted with multiplicities) enter $\underline{X}^2$. However, exactly $|X|^2$ elements enter $\underline{X}^2$ and $|X|^2<\frac{3}{2}|X|^2-|X|$ if $|X|>2$, a contradiction. Hence $|X|=2$ or $|T_{c_1}|=|X|$, and

$$\underline{X}^2=|X|e+(|X|-2)\underline{T_{c_1}}+\underline{X}. \eqno(3)$$ 
 
If for $Y\in I\setminus \{T_{c_1}\}$ we have $Y\cap E=\{q,q^2,f,f^2\}$ then $|c_1Y\cap Y|=|Y|-2$. Therefore Lemma~\ref{eq} implies that $(2)$ holds for $Y$. Thus $|T_{c_1}|=\frac{l}{2}|Y|$, where $l\geq 1$ is an integer. Note that $l<3$ because otherwise at  least $\frac{3}{2}|Y|^2-|Y|$ elements enter the element $\underline{Y}^2$ that is impossible. Moreover, $l\neq 2$, since $|T_{c_1}|\equiv 2\mod 3,~ |Y|\equiv 1\mod 3$. Thus $|Y|=2|T_{c_1}|$.

Let $3^k=max |t|,~t\in T_{c_1}$, $D_k=\{g\in D: |g|\leq 3^k\}$. Suppose that $X$ contains $l$ elements $x_1,\ldots,x_l$ of order greater than $3^k$. Then  at least $2l$ elements $qx_1,\ldots,qx_l,q^2x_1,\ldots,q^2x_l$ of order greater than $3^k$ enter the element  $\underline{X}~\underline{X}$. On the other hand, due to $(3)$ exactly $l$ elements, namely $x_1,\ldots,x_l$ of order greater than $3^k$ enter the element  $\underline{X}~\underline{X}$, a contradiction. Thus $X$ does not contain elements of order greater than $3^k$ and 

$$T_{c_1}\cup X\cup T_{c_1}X = D_k\setminus\{e\}. \eqno(4)$$ 

Suppose that $I\setminus \{T_{c_1}\}=\{X,Y,Z\}$ with $|X\cap E|=|Y\cap E|=|Z\cap E|=2$ and $|X|>2,~|Y|>2,~|Z|>2$. Then $|X|=|Y|=|Z|=|T_{c_1}|=m$. Therefore $|c_1X\cap X|=|qX\cap T_{c_1}|=m-2$ and $\underline{T_{c_1}}\underline{X}=(m-2)\underline{X}+\underline{Y}+\underline{Z}$. Due to $(4)$, we have $D_k\setminus\{e\}=T_{c_1}\cup X\cup Y \cup Z$. However, $3m+2=|D_k\setminus\{e\}|=|T_{c_1}|+|X|+|Y|+|Z|=4m$. Since $m>2$, we have a contradiction. Thus in this case at least one of the sets $X,~Y,~Z$ has cardinality $2$. 

Suppose that $I\setminus \{T_{c_1}\}=\{X,Y\}$ with $|X\cap E|=2,~|Y\cap E|=4$ and $|X|>2$. Then $|X|=|T_{c_1}|=m,~|Y|=2|T_{c_1}|=2m$, and $\underline{T_{c_1}}\underline{X}=(m-2)\underline{X}+\underline{Y}$. Again $(4)$ implies that $D_k\setminus\{e\}=T_{c_1}\cup X\cup Y \cup Z$ and we have a contradiction. Therefore in all cases there exists a basic set $X\in I\setminus \{T_{c_1}\}$ that has form $\{q,q^2\},~q\in E$. 

Let us show that $\{q,q^2\}T_{c_1}$ is a basic set. If the elements $qc_1,~q^2c_1^2,~q^2c_1,~qc_1^2$ lie in one basic set~$Y$ then $|Y|=2|T_{c_1}|$ and $Y=\{q,q^2\}T_{c_1}$. Otherwise there are  two basic sets 

$$Y=qA \cup q^2(T_{c_1} \setminus A),~Z=q(T_{c_1}\setminus A) \cup q^2A,$$ 
where $A\subseteq T_{c_1},~T_{c_1}\setminus A=A^{-1}$. Then Lemma \ref{oarg} implies that $c_1\in \rad(A\setminus C_1),~c_1\in \rad(A^{-1}\setminus C_1)$. Suppose that $yz=c_1,~y,z\in A$ and $|y|=|z|>3$. Then $z=c_1y^{-1}\in A^{-1}$, a contradiction. Thus the element $c_1$ enters the element  $\underline{Y}~\underline{Z}=(q+q^2)\underline{A}~\underline{A}^{-1}+\underline{A}^2+(\underline{A}^{-1})^2$  with coefficient $1$. In the other hand,  there exist elements from $T_{c_1}$ that enter the element $\underline{Y}~\underline{Z}$ with coefficient at least $2$, a contradiction. Therefore $X=\{q,q^2\}T_{c_1}$ is a basic set of $\mathcal{A}$ and Statement $(1)$ of Lemma \ref{Ts} holds.
\end{proof}

\begin{proof}[Proof of  Lemma \ref{Ts}]
If $\rad(T)>e$ for some $T\in I\setminus \{T_{c_1}\}$ then $C_1\leq \rad(T)$ for every $T\in I\setminus \{T_{c_1}\}$ and Statement $(5)$ holds. Thus  due to Lemma \ref{sset} we may assume that there exists nonrational $T\in I\setminus \{T_{c_1}\}$ such that $\rad(T)=\{e\}$. Then $T$  contains one, two or three elements of order $3$. Suppose that $T$ contains exactly one element $q$ of order $3$. Then $C^{T_{c_1}}_{TT^{-1}}=|c_1T\cap T|=|T|-1$. This is obvious if $T$ is regular and follows from Lemma \ref{oarg} otherwise. Lemma \ref{eq} yields that
$$|T_{c_1}|(|T|-1)=C^{T_{c_1}}_{TT^{-1}}|T_{c_1}|=C^{T}_{T_{c_1}T}|T|.$$
Thus  $|T|$ divides $(|T|-1)|T_{c_1}|$. If $|T_{c_1}|=l|T|$ where $l>1$, then at least $|T|+l|T|(|T|-1)$ elements enter the element $\underline{T}~\underline{T}^{-1}$. We have a contradiction because $|T|+l|T|(|T|-1)>|T|^2$. So either $|T|=1$ or  $|T_{c_1}|=|T|$. In the first case $T=\{q\}$ and  Statement $(4)$ of the lemma holds.

In the second case we have that $|T_{c_1}|=|T|\equiv 1~\mod~3$  by Lemma \ref{oarg}. It follows that  $T_{c_1}$ contains four elements $c_1,~c_1^2,~qc_1~,q^2c_1^2$ of order $3$. Suppose that $c_1=tx,~t\in T_{c_1},~x\in T$. If $|t|>|x|\geq 3$ then $|tx|=|t|>3$. The same is true if $|x|>|t|\geq 3$. So $|x|=|t|$. If $|t|>3$ then   $x=c_1t^{-1}\in T_{c_1}$ by Lemma~\ref{oarg}. Thus $|x|=|t|=3$. Then $c_1^2$ enters the element $\underline{T}\underline{T_{c_1}}$ whereas $c_1$ does not, a contradiction with rationality of $T_{c_1}$.

Suppose that $T$ contains exactly three elements of order $3$. We consider the case when $s,~sc_1,~s^2c_1\in T,~s^2,~s^2c_1^2,~sc_1^2\in T^{-1}$. In other cases such that $|T\cap E|=3$ the arguments are similar. Note that $|c_1T\cap T|=|T|-2$ and $(|T|-2)|T_{c_1}|$ is divisible by  $|T|$. Lemma \ref{oarg} implies that $|T|$ is divisible by $3$. However, $(|T|-2)|T_{c_1}|$ is not divisible by $3$ because $T_{c_1}$ contains two elements of order  $3$, a contradiction.

Suppose that $T$ contains exactly two elements of order $3$. 

\textbf{Case 1.} Let $s,~s^2c_1^2\in T,~s^2,~sc_1\in T^{-1}$. Denote the basic set containing $sc_1^2$ by $X$. Note that $c_1$ and $c_1^2$ appear in $\underline{X}~\underline{T}$ only as the product of two elements of order $3$ because otherwise $X\cap T^{-1}\neq \varnothing$ by Lemma \ref{oarg}. So $c_1$ enters the element $\underline{X}~\underline{T}$ whereas $c_1^2$ does not, a contradiction with $T_{c_1}=T_{c_1^2}$.  

\textbf{Case 2.} Let $s,~sc_1\in T,~s^2,~s^2c_1^2\in T^{-1}$. Then  $s^2c_1\in X$ since otherwise $X$ contains exactly one element of order $3$, hence, $|X|=1$, and Statement $(4)$ of the lemma holds. The element $sc_1^2$ enters the element $\underline{T_{c_1}}~\underline{T}$ as the product of $s$ and $c_1^2$. Since $sc_1^2,s^2c_1\in X$, the element $s^2c_1$ also enters $\underline{T_{c_1}}~\underline{T}$. Thus  there are elements  $r\in T_{c_1}$ and $t\in T$ such that $|r|>3,~|t|>3$, and $rt=c_1s^2$. From Lemma~\ref{oarg} it follows that $r^{-1}c_1\in T_{c_1}$. By Lemma~\ref{Tc1}, either $T_{c_1}\cup\{e\}\leq D$ or $T_{c_1}\cup\{e,~q,~q^2\}\leq D,~|q|=3$.  The latter case is impossible because then $T_{c_1}=T$.  So $K=T_{c_1}\cup\{e\}$ is an $\mathcal{A}$-subgroup. We have that $|sK\cap T|\equiv 1~\mod~3$ and $|s^2K\cap T|\equiv 0~\mod~3$. Thus  Lemma \ref{intersection} yields that  $s^2K\cap T=\varnothing$. Otherwise $t=r^{-1}c_1s^2\in s^2K\cap T$, a contradiction. 

Other cases such that $|T\cap E|=2$ are similar to Case $1$ or to Case $2$. We conclude that $T$ does not contain exactly two elements of order $3$. 
\end{proof}

\begin{corl} \label{R}
Let $3^k=max |t|,~t\in T_{c_1}$, $D_k=\{g\in D: |g|\leq 3^k\}$, $T_{c_1}$ be nonregular, rational, and $R$ be the union of all basic sets containing elements of order $3$. Then $(D_k\setminus \{e\})\subseteq R$. Moreover, if $\rad(T)=\{e\}$ for every $T\in I\setminus \{T_{c_1}\}$ then $R=D_k\setminus \{e\}$.
\end{corl}

\begin{proof}
The statement of the corollary is clear if one of  Statements $(1)-(4)$ of Lemma \ref{Ts} holds. Suppose that  Statement $(5)$ of Lemma \ref{Ts} holds. Then $K=T_{c_1}\cup \{e\}$ is a cyclic group. Let $T\in I\setminus \{T_{c_1}\}$ and $q\in T\cap E$. Since $\rad(T)$ is an $\mathcal{A}$-subgroup, we conclude that $T_{c_1}=K\setminus \{e\}\subseteq \rad(T)$. Thus $qK\subseteq T$, $q^2K\subseteq T^{-1}$, and $R=T_{c_1}\cup T\cup T^{-1}\supseteq T_{c_1}\cup sK\cup s^2K=D_k\setminus \{e\}$.
\end{proof}

\section{$S$-rings over $D=\mathbb{Z}_3\times \mathbb{Z}_{3^n}$: nonregular case }

The purpose of this section is to describe nonregular $S$-rings with trivial radical over $D$. The main result  can be formulated as follows.

\begin{theo} \label{nonreg}
Let $\mathcal{A}$ be an $S$-ring over $D$. Suppose that $\rad(\mathcal{A})=e$. Then one of the following statements holds:

$(1)$ $\mathcal{A}$ is regular;

$(2)$ $\mathcal{A}=\mathcal{A}_H\otimes\mathcal{A}_L$, where $rk(\mathcal{A}_H)=2$ and $|L|\leq 3 \leq |H|$. In particular,  $\mathcal{A}$ is schurian.

\end{theo}

The proof of Theorem \ref{nonreg} will be given in the end of the section.

\begin{lemm} \label{nonregc1}
Let $T$ be a nonregular basic set  of an $S$-ring $\mathcal{A}$ over $D$. Suppose that $C_1$ is an $\mathcal{A}$-subgroup. Then $C_1\leq \rad(T)$.
\end{lemm}

\begin{proof}
Let $3^m=min |t|,~t\in T$. Lemma \ref{oarg} implies that $C_1 \leq \rad(T\setminus T_m)$. Suppose that there exists  $t\in T_m$ such that $tc_1\notin T$. Then $X=T_{tc_1}$ is other than $T$. Let $\pi:D\rightarrow D/C_1$ be the quotient epimorphism. The sets $\pi(T)$ and $\pi(X)$ are basic sets of $\mathcal{A}_{D/C_1}$ and $\pi(t)\in\pi(T)\cap\pi(X)$. Therefore $\pi(T)=\pi(X)$. So there is an element $y\in T\setminus T_m$ such that either $yc_1 \in X$ or $yc_1^2 \in X$. However, $yc_1,~yc_1^2\in T$ by Lemma \ref{oarg}, a contradiction. Thus for every $t\in T_m$ we have $tc_1\in T$ and $C_1\leq \rad(T)$.
\end{proof}

The key point of the proof of Theorem \ref{nonreg} is the following statement.

\begin{lemm} \label{nonregset}
Let $T$ be a nonregular basic set  of an $S$-ring $\mathcal{A}$ over $D$. Suppose that $T$ does not contain  elements of order $3$. Then $\rad(T)>e$.
\end{lemm}

\begin{proof}
Assume the contrary. Let $3^m=min |t|,~t\in T$. Then $m>1$, $c_1\notin \rad(T_m)$, $T_m\neq \varnothing$, and $T \setminus T_m\neq \varnothing$. From Lemma \ref{lowset} it follows that $|T_m|\leq 6$. By Lemma \ref{nonregc1}, the group $C_1$ is not an $\mathcal{A}$-subgroup. Lemma \ref{Tc1} yields that either $T_{c_1}$ is rational and nonregular or  $E$ is an $\mathcal{A}$-subgroup. Let us define  groups $K$ and $M$ as in the proof of Lemma \ref{lowset}: 

$$K=\{\sigma_m:m\in \mathbb{Z}^{*}_{3^n}\},~\sigma_m: x\rightarrow x^m,~M=K_{\{T\}}.$$ 

\begin{prop}\label{orbits}
In the above notation $|T\setminus T_m|\geq |T_m|$, and the equality holds if and only if

$(1)$ $T_m$ is a union of three  $M$-orbits;

$(2)$ $T\setminus T_m$ is an $M$-orbit;

$(3)$ any element in $T\setminus T_m$ is of order $3^{m+1}$.

\end{prop}

\begin{proof}
Note that $M=\left(\mathbb{Z}^{*}_{3^n}\right)_{\left\{T_m\right\}}$. For every $t\in T$ we have $K_t\leq M$ since $T$ is a basic set of $\mathcal{A}$.  Therefore

$$M_t=K_t=\{1+|t|k: k = 0,\ldots,\frac{3^n}{|t|}-1\}.$$
Thus $|M_t|=\frac{3^n}{|t|}$. It follows that if $|t_1|=|t_2|$ then $|t_1M|=|t_2M|$. Let $x\in T_m$. Then  $|T_m|=|T_m/M||xM|$. Let $z\in T\setminus T_m,~ |z|>|x|$. Then 
$$|M_z|=|M_x|\frac{|x|}{|z|}.$$
Taking into account that $T_m$ is a disjoint union of at most three $M$-orbits, we obtain
that $|M_z|\leq \frac{|M_x|}{3}$ and 

$$|T|-|T_m|\geq |zM|\geq 3|xM|\geq|T_m/M||xM|=|T_m|.$$
as required. Since the equality holds only if the second and third inequalities in the above formula are equalities, we are done.
\end{proof}

Let $H=\langle T_{c_1} \rangle$ be a group of exponent  $3^k$ and let $R$ be the union of all basic sets containing elements of order $3$. Then $R$ is a rational $\mathcal{A}$-set and $R\cap T=\varnothing$. Moreover, $k<m$: if $T_{c_1}$ is rational and nonregular then  this   follows  from Corollary \ref{R}; otherwise  $T_{c_1}\subseteq E$ and, hence, $k=1<m$. So $|tk|=|t|$ for every $t\in T$ and every $k\in H$. This implies that $tH\cap T\subseteq T_m$ for every  $t\in T_m$. Therefore $T_m$ is a disjoint union of some sets $tH\cap T$ with such $t$. However, by Lemma \ref{intersection}, the number $\lambda:=|tH\cap T|$ does not depend on the choice of $t\in T$. Thus $\lambda$ divides $|T_m|$.  From Lemma \ref{lowset} it follows that $|T_m|\in\{1,2,3,4,6\}$ and, hence, 
$$\lambda\in\{1,2,3,4,6\}.$$ 

Let us prove that $T_{c_1}$ is regular. Assume the contrary. Then we have that $H=T_{c_1}\cup \{e\}$ or $H=T_{c_1}\cup \{e,q,q^2\},~q\in E\setminus C_1$ by Lemma \ref{Tc1}. In the first case $C^{T}_{TT_{c_1}}=\lambda-1$. In the second there exists $t\in T$ such that $qt\notin T$ or $q^2t\notin T$; otherwise $tq,tq^2\in T$ for every $t\in T$ and, hence, $q\in \rad(T)>e$, a contradiction. So $C^{T}_{TT_{c_1}}=\lambda-1$ or $C^{T}_{TT_{c_1}}=\lambda-2$. Lemma \ref{eq} yields that    

$$C^{T}_{TT_{c_1}}|T|=C^{T_{c_1}}_{TT^{-1}}|T_{c_1}|=|c_1T\cap T||T_{c_1}|=(|T|-\alpha)|T_{c_1}|, ~\eqno(5)$$
where $\alpha=|T_m|-|c_1T_m\cap T_m|$. We conclude that $\alpha>0$  since $c_1\notin\rad(T_m)$. Proposition \ref{orbits} implies that $|T|-\alpha >0$. The number $|T_{c_1}|-C^{T}_{TT_{c_1}}$ is not equal to $0$ because otherwise $\alpha=0$. Therefore from $(5)$ and Proposition \ref{orbits} it follows that 

$$\frac{\alpha|T_{c_1}|}{|T_{c_1}|-C^{T}_{TT_{c_1}}}=|T|\geq 2|T_m|\geq 2\alpha.$$
Thus $|T_{c_1}|\leq 2C^{T}_{TT_{c_1}}\leq 10$ because $C^{T}_{TT_{c_1}}\leq \lambda-1\leq 5$. So $|H|=9$ and, hence, $H=E$ or $H$ is a cyclic group. Let us show that the latter case is impossible. If $H$ is a cyclic group then we may assume that $H\leq C$. On the one hand, $|tH\cap T|\geq 3$ for $t\in T\setminus T_m$ by Lemma~\ref{oarg}. On the other hand, Lemma~\ref{lowset} implies that $0<|tH\cap T|\leq |tC\cap T_m|\leq 2$ for any $t\in T_m$. We have a contradiction with Lemma \ref{intersection}. Therefore $H=E$ and $T_{c_1}$ is regular.

From Lemma \ref{oarg} it follows that $tC_1\subseteq tH\cap T$ for every $t\in T\setminus T_m$. Thus $\lambda=3$ or $\lambda=6$. To complete the proof of Lemma \ref{nonregset}, we show  that both of these cases are impossible. Since $\lambda$ divides $|T_m|$ and $|T_m|\le 6$, we have $|T_m|=3$ or $|T_m|=6$. Besides, Lemma \ref{lowset} yields that $|c_1T_m\cap T_m|=0$ (otherwise $C_1\leq \rad(T)$) and $\alpha=|T_m|$. Suppose that $\lambda=6$. Then $|tH\cap T_m|=6$ for any $t\in T_m$. Hence $T_m\subseteq tH$ and $|c_1T_m\cap T_m|>0$, a contradiction. So $\lambda=3$. Then the regularity of  $T_{c_1}$ implies that $|T_{c_1}|\in\{2,3,4,6,8\}$. If $|T_{c_1}|\in\{4,6,8\}$ then $T_{c_1}$ is rational and $C^{T}_{TT_{c_1}}=\lambda-1=2$ because for any $t\in T\setminus T_m$  we have $c_1t,c_1^2t\in tT_{c_1}\cap T$ by Lemma \ref{oarg}. If $|T_{c_1}|\in\{2,3\}$ then $C^{T}_{TT_{c_1}}=\lambda-2=1$ because $t,c_1^2t\in tH\cap T$ but $t,c_1^2t\notin tT_{c_1}\cap T$ for any  $t\in T\setminus T_m$. On the other hand, from $(5)$ it follows that 

$$|T|=\alpha+\frac{2\alpha}{|T_{c_1}|-2}\leq 2\alpha$$
for $|T_{c_1}|\in\{4,6,8\}$ and 

$$|T|=\alpha+\frac{\alpha}{|T_{c_1}|-1}\leq 2\alpha$$ 
for $|T_{c_1}|\in\{2,3\}$. We  conclude that $|T|=2\alpha=2T_m\in \{6,12\}$ since $|T|\geq 2\alpha$  by Proposition \ref{orbits}. 

Let $\pi:D\rightarrow D/E$ be the quotient epimorphism and $T^{'}=\pi(T)$. Then $T^{'}$ is a nonregular basic set with trivial radical of a circulant $S$-ring over $D/E$. Lemma \ref{cyclering} implies that $|T^{'}|\geq 4$. On the other hand, $|T^{'}|=\frac{|T|}{\lambda}=\frac{|T|}{3}$ by the definition of $\lambda$. Thus $|T|\neq 6$. Therefore $|T|=12$. Then $|T^{'}|=4$. Since $T^{'}$ is a nonregular set with trivial radical, Lemma \ref{cyclering} implies that $T^{'}=\langle T^{'} \rangle \setminus \rad(T^{'})=\langle T^{'} \rangle \setminus \{e\}$ and $4=|T^{'}|=3^{i}-1$ for some integer $i$, a contradiction.
\end{proof}

\begin{prop} \label{nonreghigh}
Suppose that there exists a nonregular highest basic set $X$ of an $S$-ring $\mathcal{A}$ over $D$ such that $\rad(X)=e$. Then Statement $(2)$ of Theorem \ref{nonreg} holds. In particular, $\rad(\mathcal{A})=e$.
\end{prop}

\begin{proof}
Since $X$ is nonregular and $\rad(X)=e$, we conclude by Lemma \ref{nonregset}  that the set $X\cap E$ is not empty. So neither $C_1$ nor $E$ is an $\mathcal{A}$-subgroup (the former follows from Lemma \ref{nonregc1}). So  we have  that $T_{c_1}$ is nonregular and rational by Statement $(2)$ of Lemma \ref{Tc1}. 

Show that $T_{c_1}$ is a highest basic set. This is obvious if $T_{c_1}=X$. If $T_{c_1}\neq X$ then every basic set $T$ with $T\cap E\neq \varnothing$ and $T\neq T_{c_1}$ has trivial radical (otherwise $c_1\in \rad(X)$). From Corollary \ref{R} it follows that $D_k\setminus \{e\}=R$, where $3^k=max |t|,~t\in T_{c_1}$, and $R$ is a union of all basic sets containing elements of order $3$. Since $X\subseteq R$ and $X$ is a highest set, this implies that $D_k=D$ and we are done. 

Lemma \ref{Tc1} implies that $H=T_{c_1}\cup \{e\}$ is an $\mathcal{A}$-subgroup or $H=T_{c_1}\cup \{e,q,q^2\}$ is an $\mathcal{A}$-subgroup, where $q\in E\setminus C_1$. In the latter case $H=D$ because $T_{c_1}$ is highest. Thus  $T_{c_1}=X$ is the unique highest basic set and $q\in \rad(T_{c_1})$ that is impossible. Thus $H=T_{c_1}\cup \{e\}$. Now, if $H$ is noncyclic then $H=D$ and $\mathcal{A}=\mathcal{A}_H\otimes\mathcal{A}_L$ for $L=\{e\}$. If $H$ is cyclic then  we have $\mathcal{A}=\mathcal{A}_H\otimes\mathcal{A}_L$ for $L$ of order $3$ by Statements $(1)$ and $(4)$ of Lemma \ref{Ts}. 
\end{proof}

\begin{proof}[Proof of the Theorem \ref{nonreg}]
Suppose that $\mathcal{A}$ is not regular. Then there exists a nonregular highest basic set $X$ and we are done by Proposition \ref{nonreghigh}.
\end{proof}

\section{$S$-rings over $D=\mathbb{Z}_3\times \mathbb{Z}_{3^n}$: regular case }
In this section we classify regular $S$-rings with trivial radical over $D$. The main result is given by the following theorem.

\begin{theo} \label{regular}
Let $\mathcal{A}$ be a regular $S$-ring over $D$. Suppose that $\rad(\mathcal{A})=e$. Then $\mathcal{A}$ is cyclotomic. Moreover,  $\mathcal{A}$ is Cayley isomorphic to $Cyc(K,D)$, where $K\leq \operatorname{\aut}(D)$ is one of the groups listed in Table $1$. In particular, $\mathcal{A}$ is schurian.
\end{theo}

\begin{center}
{\small
\begin{tabular}{|l|l|l|}
  \hline
  name & generators & size    \\
  \hline
  $K_0$ & $(x,s)\rightarrow (x,s)$ & $1$ \\ \hline
  $K_1$ & $(x,s)\rightarrow (x,s^2)$  & $2$ \\  \hline
  $K_2$ & $(x,s)\rightarrow (x^{-1},s)$  & $2$ \\  \hline
  $K_3$ & $(x,s)\rightarrow (x^{-1},s),~(x,s)\rightarrow (x,s^2)$ & $4$ \\ \hline
  $K_4$ & $(x,s)\rightarrow (x^{-1},s^2)$ & $2$  \\ \hline
  $K_5$ & $(x,s)\rightarrow (sx^{-1},s)$  & $2$  \\  \hline
  $K_6$ & $(x,s)\rightarrow (sx,sc_1)$  & $3$  \\  \hline
  $K_7$ & $(x,s)\rightarrow (sx,sc_1),~(x,s)\rightarrow (x,s^2c_1)$  & $6$  \\  \hline
	$K_8$ & $(x,s)\rightarrow (sx,sc_1^2),~(x,s)\rightarrow (x^{-1},sc_1)$  & $6$ \\  \hline
  $K_9$ & $(x,s)\rightarrow (sx,sc_1^2),~(x,s)\rightarrow (x^{-1},s^2)$  & $6$  \\ \hline

\end{tabular}
}

Table 1. The groups of cyclotomic rings with trivial radical.

\end{center}

Before we start to prove Theorem \ref{regular}, we formulate a technical lemma on $S$-rings over $E=S\times C_1$, where $S=\langle s \rangle,~C_1=\langle c_1 \rangle$.

\begin{lemm}\label{sringE}
Let  $\mathcal{A}$ be an $S$-ring over $E$. Suppose that  $C_1$ is an $\mathcal{A}$-subgroup. Then $\mathcal{A}$ is one (up to Cayley isomorphism) of the following $S$-rings:
\begin{enumerate}
\item $\mathcal{A}=\mathbb{Z}C_1\wr \mathcal{A}_S$, where $rk(\mathcal{A}_S)=2$;

\item $\mathcal{A}=\mathbb{Z}E$;

\item $\mathcal{A}=\mathbb{Z}C_1\wr \mathbb{Z}S$;

\item $\mathcal{A}=\mathbb{Z}C_1\otimes \mathcal{A}_S$, where  $rk(\mathcal{A}_S)=2$;

\item $\mathcal{A}=\mathcal{A}_{C_1}\otimes \mathbb{Z}S$, where $rk(\mathcal{A}_{C_1})=2$;

\item  $\mathcal{A}=\mathcal{A}_{C_1}\wr \mathbb{Z}S$, where $rk(\mathcal{A}_{C_1})=2$;

\item $\mathcal{A}=Cyc(M,E)$, where $M=\{e,~\delta\},~\delta:x\rightarrow x^{-1}$;

\item $\mathcal{A}=\mathcal{A}_{C_1}\otimes \mathcal{A}_S$, where $rk(\mathcal{A}_Q)=rk(\mathcal{A}_{C_1})=2$;

\item  $\mathcal{A}=\mathcal{A}_{C_1}\wr \mathcal{A}_S$,where $rk(\mathcal{A}_S)=rk(\mathcal{A}_{C_1})=2$;
\end{enumerate}

and  $\mathcal{S}(\mathcal{A})$ without $e$ is one (up to Cayley isomorphism) of the following  forms :

\begin{enumerate}
\item $\mathcal{S}(\mathcal{A})=\{\{c_1\},~\{c_1^2\},~\{s,sc_1,sc_1^2,s^2,s^2c_1,s^2c_1^2\}\}$;

\item $\mathcal{S}(\mathcal{A})=\{\{c_1\},~\{c_1^2\},~\{s\},~\{s^2\},~\{sc_1\},~\{s^2c_1\},~\{sc_1^2\},~\{s^2c_1^2\}\}$;

\item $\mathcal{S}(\mathcal{A})=\{\{c_1\},~\{c_1^2\},~\{s,sc_1,sc_1^2\},~\{s^2,s^2c_1,s^2c_1^2\}\}$;

\item $\mathcal{S}(\mathcal{A})=\{\{c_1\},~\{c_1^2\},~\{s,s^2\},~\{sc_1,s^2c_1\},~\{sc_1^2,s^2c_1^2\}\}$;

\item $\mathcal{S}(\mathcal{A})=\{\{c_1,c_1^2\},~\{s\},~\{s^2\},~\{sc_1,sc_1^2\},~\{s^2c_1,s^2c_1^2\}\}$;

\item  $\mathcal{S}(\mathcal{A})=\{\{c_1,c_1^2\},~\{s,sc_1,sc_1^2\},~\{s^2,s^2c_1,s^2c_1^2\}\}$;

\item $\mathcal{S}(\mathcal{A})=\{\{c_1,c_1^2\},~\{s,s^2\},~\{sc_1,s^2c_1^2\},~\{sc_1^2,s^2c_1\}\}$;

\item $\mathcal{S}(\mathcal{A})=\{\{c_1,c_1^2\},~\{s,s^2\},~\{sc_1,s^2c_1^2,sc_1^2,s^2c_1\}\}$;

\item  $\mathcal{S}(\mathcal{A})=\{\{c_1,c_1^2\},~\{s,s^2,sc_1,s^2c_1^2,sc_1^2,s^2c_1\}\}$.
\end{enumerate}

\end{lemm} 

\begin{proof}
Follows from calculations in the group ring of $E$ that made by \cite{GAP}.
\end{proof}

\begin{proof}[Proof of the Theorem \ref{regular}]
Let $X$ be a highest basic set and $x\in X$. Without loss of generality we may assume that $\langle x \rangle=C$. From Lemma \ref{lowset} it follows that $|X|\in \{1,2,3,4,6\}$. Let us consider all these cases. 

\textbf{Case 1: $\boldsymbol{|X|=1}$}. In this case $X=\{x\}$. Lemma \ref{radbasic} implies that  $C$ is an $\mathcal{A}$-subgroup and $\mathcal{A}_{C}=\mathbb{Z}C$. Suppose that the basic set $T_s$  that contains $s$ is not regular. Then from Lemma  \ref{nonregc1}  it follows that $C_1\leq \rad(T_s)$. Hence $|T_s|\geq 6$. Let $Y$ be the highest basic set containing $sx^{-1}$. Suppose that $|Y|=6$. Then $Y_{0e}\neq \varnothing,~Y_{1s}\neq \varnothing,~Y_{2s^2}\neq \varnothing,$ and $Y=Y^{-1}$. Therefore every highest basic set is rationally conjugate to $Y$ that is not true for $X$. Thus $|Y|\leq 4$ . By Lemma~\ref{aset}, the set $x^{-1}Y$ is a basic set containing $s$. So $x^{-1}Y=T_s$. However,  $|x^{-1}Y|\leq 4$ and  $|T_s|\geq 6$, a contradiction. Therefore $T_s$ is regular. This implies that $E$ is an $\mathcal{A}$-subgroup. Since also $c_1,c_1^2\in \mathcal{A}$, we conclude that $\mathcal{S}(\mathcal{A}_E)$ is  one (up to Cayley isomorphism) of the forms $(1)-(4)$ from Lemma \ref{sringE}. 

If  $\mathcal{S}(\mathcal{A}_E)$ is of the  form $(1)$  then the set $X_1=x\{s,sc_1,sc_1^2,s^2,s^2c_1,s^2c_1^2\}$ is a highest basic set such that $c_1\in \rad(X_1)$, a contradiction. If $\mathcal{S}(\mathcal{A}_E)$ is of the  form $(3)$ the set $X_1=x\{s,sc_1,sc_1^2\}$ is a highest basic set such that $c_1\in \rad(X_1)$, a contradiction. If $\mathcal{S}(\mathcal{A}_E)$ is of the  form $(2)$ or $(4)$ then $\mathcal{A}_C=\mathbb{Z}C$ and $S$ is an $\mathcal{A}$-subgroup. Therefore $\mathcal{A}$ is the tensor product of $\mathcal{A}_C$ and $\mathcal{A}_S$. In the former case $\mathcal{A}=\mathbb{Z}D=Cyc(K_0,D)$; in the latter case  $\mathcal{A}=Cyc(K_1,D)$.

\textbf{Case 2: $\boldsymbol{|X|=2}$}. Let $X=\{x,x_1\}$. If $x_1\notin C$ then without loss of generality we may assume that $x_1\in sC$. If $x_1\in C$ then Lemma \ref{lowset} implies that $x_1=x^{-1}$. In the first case put $y=s^2x_1$. Note that $y\in C$.

Let  $X=\{x\}\cup s\{y\}$. Since $C$ is cyclic, there exists $k\in \mathbb{Z}$ such that $y=x^k$. By Theorem \ref{burn}, the set $X^{(2)}=\{x^2\}\cup s^2\{y^2\}$ is basic. Since $2\underline{sxy}=\underline{X}^2-\underline{X}^{(2)}$, we conclude that $Y=\{sxy\}\in \mathcal{S}(\mathcal{A})$. If $k\equiv 1\mod 3$ then $Y$ is a highest basic set of cardinality $1$, and we are done by the previous case. So we may assume that $k\equiv 2\mod 3$. Theorem \ref{burn} yields that the set $Z=X^{(-k)}=\{x^{-k}\}\cup s\{x^{-k^2}\}$ is basic. 

Consider the element 

$$\xi=\underline{X}~\underline{Z}=x^{-k+1}+sx^{-k^2+1}+s+s^2x^{-k^2+k}.~\eqno(6)$$ 
The elements $x^{-k+1}$ and $s^2x^{-k^2+k}$ have order $3^n$ because $k\equiv 2\mod 3$. Hence $T_s=\{s\}$ or $T_s=\{s,sx^{-k^2+1}\}$. In the latter case $T_s$ is regular because otherwise $|T_s|\geq 4$ by Lemma \ref{oarg}. If $T_s=\{s,sx^{-k^2+1}\}$ then $2\underline{s^2x^{-k^2+1}}=\underline{T_s}^2-\underline{T_s}^{(2)}$ and $\{s^2x^{-k^2+1}\}\in \mathcal{S}(\mathcal{A})$. Therefore there exists a basic set that has form $\{q\},~q\in E\setminus C_1$. Without loss of generality we may assume that $T_s=\{s\}$. Then the sets 
$$sX=s\{x\}\cup s^2\{y\},~s^2X=s^2\{x\}\cup \{y\}$$
are basic. The latter implies that $X^{(k)}=s^2X$. Thus $x^{k^2}=y^k=x$ and, hence, $3^n$ divides $k^2-1=(k-1)(k+1)$. So  $3^n$ divides $k+1$. This shows that $y=x^{-1}$. Since $C$ is cyclic, Theorem \ref{burn} yields that every highest basic set is rationally conjugate to $X$ or to $sX$. Hence every highest basic set has one of the following three forms 
$$\{x\}\cup s\{x^{-1}\},~s\{x\}\cup s^2\{x^{-1}\},~s^2\{x\}\cup \{x^{-1}\},~x\in C,~|x|=3^n.$$ 
From Theorem \ref{sch} it follows  that $\{u,u^{-1}\}$ is a basic set for every $u\in C,~ |u|<3^n$. Since $\{s\}\in \mathcal{S}(\mathcal{A}),$ the sets $s\{u,u^{-1}\},~s^2\{u,u^{-1}\}$ are basic for every $u\in C,~ |u|<3^n$. Thus we have $\mathcal{A}=Cyc(K_5,D)$.

Now let $X=\{x,x^{-1}\}$. Then $C$ is an $\mathcal{A}$-subgroup by Lemma \ref{radbasic}. The basic sets of $\mathcal{A}_{C}$ are of the form $\{x^k,x^{-k}\},~k\in \mathbb{Z},$ by Theorem \ref{burn} and Theorem \ref{sch}. Let us show that  
$$\rad(T_s)=e.~\eqno(7)$$ 
Assume the contrary. Then  $sc_1,sc_1^2\in T_s$. Let $Y$ be the highest basic set containing $sx$. Note that $|Y|\in \{2,4\}$ because otherwise Theorem~\ref{burn} implies that every highest basic set has cardinality $3$ or $6$. So $Y=s\{x\}\cup s^2\{x^{-1}\}$ or $Y=s\{x,y\}\cup s^2\{x^{-1},y^{-1}\},~y\in C$. The element $s$ enters the element $\psi=\underline{X}~\underline{Y}$. Therefore $sc_1$ and $sc_1^2$ enter $\psi$. If $|Y|=2$ then only the elements $s$ and $sx^2$ from $sC$ enter $\psi$, a contradiction. If $|Y|=4$ then only the elements $s,sx^2,sxy,sx^{-1}y$ from $sC$ enter $\psi$. However, $sx^2$ and one of the elements $sxy,sx^{-1}y$ have order $3^n$, a contradiction. The obtained contradiction proves $(7)$. Moreover,  $T_s$ is regular since otherwise  we have $C_1\leq \rad(T_s)$ by Lemma \ref{nonregc1}. Hence $E=\langle T_{c_1},T_s \rangle$ is an $\mathcal{A}$-subgroup by Lemma \ref{radbasic} and $\mathcal{S}(\mathcal{A}_E)$ is one (up to Cayley isomorphism) of the forms $(5),(7),(8)$ from Lemma \ref{sringE}. 

If $\mathcal{S}(\mathcal{A}_E)$ is of the form $(5)$ then $\mathcal{A}_S=\mathbb{Z}S$ and $C$ is an $\mathcal{A}$-subgroup. Therefore, $\mathcal{A}=\mathcal{A}_C\otimes\mathcal{A}_S$ and $\mathcal{A}=Cyc(K_2,D)$. If $\mathcal{S}(\mathcal{A}_E)$ is of the form $(7)$ or $(8)$ then the set $\{s,s^2\}$ is basic and, hence, every basic  set outside $C$ is of the form either $s\{y\}\cup s^2\{y^{-1}\}$ or $s\{y,y^{-1}\}\cup s^2\{y,y^{-1}\}$, where $y\in C$. Assume that there exist basic sets 
$$Y=s\{y\}\cup s^2\{y^{-1}\},~y\in C,~Z=s\{z,z^{-1}\}\cup s^2\{z,z^{-1}\},~z\in C.$$ 
Then the elements $sz$ and $s^2z^{-1}$ enter the element $(zy^{-1}+yz^{-1})\underline{Y}$ whereas the elements $sz^{-1}$ and $s^2z$ do not, a contradiction. So all basic sets outside $C$   are of the form $s\{y\}\cup s^2\{y^{-1}\},~y\in C$, or all basic sets outside $C$ are of the form $s\{y,y^{-1}\}\cup s^2\{y,y^{-1}\},~y\in C$. Therefore $\mathcal{A}=Cyc(K_4,D)$ or $\mathcal{A}=\mathcal{A}_C \otimes \mathcal{A}_{S}=Cyc(K_3,D) $, where $\mathcal{A}_{S}$ is an $S$-ring of rank $2$ over $S$.  

\textbf{Case 3: $\boldsymbol{|X|=4}$}. Lemma \ref{lowset} implies that $X=\{x,x^{-1}\}\cup s\{y\}\cup s^2\{y^{-1}\},~y\in C,$ because $x\in C$. Obviously,
   
$$\underline{X}~\underline{X}=x^2+x^{-2}+s^2y^2+sy^{-2}+4e+2sxy+2s^2xy^{-1}+2sx^{-1}y+2s^2x^{-1}y^{-1}.~\eqno(8)$$

Since   $X^{(2)}\in \mathcal{S}(\mathcal{A})$ and exactly two elements from $sxy,~s^2xy^{-1},~sx^{-1}y,~s^2x^{-1}y^{-1}$ have order $3^n$, there exists a highest basic set $Y$ such that $|Y|\leq 2$ and we arrive at Case $1$ or $2$.

\textbf{Case 4: $\boldsymbol{|X|=3}$}. In this case $X=\{x\}\cup s^i\{y\}\cup s^j\{z\},~x,y,z\in C,~i,j\in \{1,2\}$. Since $C$ is cyclic, there exist $k,l\in \mathbb{Z}$ such that $y=x^k,z=x^l$. Note that $T=\{s^ix^{k+1},s^jx^{l+1},s^{i+j}x^{k+l}\}$ is an $\mathcal{A}$-set because
$$s^ix^{k+1}+s^jx^{l+1}+s^{i+j}x^{k+l}=\underline{X}^2-\underline{X}^{(2)}.$$
If $k\equiv 2 \mod 3$ or $l\equiv 2 \mod 3$ then $T$ contains exactly one element of order $3^n$ and we arrive at Case~$1$. So we assume that $k\equiv 1 \mod 3$ and $l\equiv 1 \mod 3$. Suppose that $i=j$. Then $s^iy,s^jz\in \langle s^ix \rangle$. From Lemma \ref{lowset} it follows that $s^jz=(s^iy)^{-1}$ that is not true. Therefore $i\neq j$. Without loss of generality we  assume that $i=1,~j=2$.

Let us show that 
$$\underline{C_1},~\underline{E}\in \mathcal{A}.~\eqno(9)$$ 
Suppose first that $|xE\cap X|=3$. Then $X=x\{e,s\varepsilon_1,s^2\varepsilon_2\},~\varepsilon_1,\varepsilon_2\in C_1$. Note that $\varepsilon_2\neq \varepsilon_1^{-1}$ since otherwise $\rad(X)=C_1>e$.  A straightforward computation shows that
$$\underline{E\setminus C_1}=\underline{X}~\underline{X}^{-1}-3e.$$
So $E\setminus C_1$  is an $\mathcal{A}$-set. Hence $E=\langle E\setminus C_1 \rangle$ is an $\mathcal{A}$-subgroup by Lemma \ref{radbasic}  and $C_1=E\setminus(E\setminus C_1)$ is an $\mathcal{A}$-subgroup. If $|xE\cap X|<3$ then from Theorem \ref{sch} it follows that $X^{[3]}=\{x^3,y^3,z^3\}\subseteq C$ is an $\mathcal{A}$-set.  Suppose that $\{x^3\}$ is a basic set. Then the set $x^3X^{(-2)}$ is basic. Hence $x^3X^{(-2)}=X$ and $|xE\cap X|=3$, a contradiction. Similarly, neither $y^3$ nor $z^3$ is a basic set. Therefore $X^{[3]}$ is a basic set. Theorem~\ref{cyclering} yields that $X^{[3]}\in Orb(K,C_{n-1})$ for some $K\leq \aut(C_{n-1})$. If $|X^{[3]}|=2$ then $\rad(X^{[3]})=e$ and, by Lemma \ref{cycleorbit}, without loss of generality we have $x^{3k}=y^3=x^3,~x^{3l}=z^3=x^{-3}$. This contradicts the fact that  $k\equiv 1 \mod 3$ and $l\equiv 1 \mod 3$. Therefore $|X^{[3]}|=3$ and  this set has nontrivial radical by Lemma~\ref{cycleorbit}. Thus $X^{[3]}=x^3C_1$ and  $C_1=\rad(X^{[3]})$ is an $\mathcal{A}$-subgroup by Lemma~\ref{radbasic}. To prove that $E$ is an $\mathcal{A}$-subgroup it is sufficient to verify that $T_s$ is regular. Indeed, if $T_s$ is regular then $E=\langle T_s,C_1\rangle$ is an $\mathcal{A}$-subgroup by Lemma \ref{radbasic}. Assume that $T_s$ is not regular. Then $C_1\leq \rad(T_s)$ by Lemma \ref{nonregc1} and, hence, $|T_s|\geq 6$. Since $X^{[3]}=x^3C_1$ is a basic set, Lemma \ref{radbasic} implies that $C_{n-1}=\langle X^{[3]} \rangle$ is an $\mathcal{A}$-subgroup. Let $\pi:G\rightarrow G/C_1$ be the quotient epimorphism. Then $\{\pi(x^3)\}$ is a highest basic set of $\mathcal{A}_{C_{n-1}/C_1}$ and, hence, $\mathcal{A}_{C_{n-1}/C_1}=\mathbb{Z}{(C_{n-1}/C_1)}$ by Lemma \ref{aset}. The set $\pi(T_s)$ is not regular. Therefore Lemma~\ref{nonregc1} implies that $\rad(\pi(T_s))>e$ and $|\pi(T_s)|\geq 6$. So $|T_s|\geq 18$. Exactly $9$ elements enter the element $\theta_1=\underline{X}~\underline{Y}$, where $Y$ is the highest basic set containing $y^{-1}$. On the other hand, $s$ enters $\theta_1$ and we conclude that at least $18$ elements enter $\theta_1$, a contradiction. Thus $T_s$ is regular and $E$ is an $\mathcal{A}$-subgroup.

Now we prove that 
$$\{c_1\}\in \mathcal{S}(\mathcal{A}).$$ 
Without loss of generality we may assume that $x^{3^{n-1}}=c_1$. Since $k\equiv 1 \mod 3$ and $l\equiv 1 \mod 3$, we have $y^{3^{n-1}}=z^{3^{n-1}}=c_1$ and the set $X^{(3^{n-1}+1)}=c_1X$ is basic. The element $c_1$ enters the element $\underline{c_1X}~\underline{X}^{-1}$ whereas $c_1^2$ does not. So $c_1$ and $c_1^2$ lie in the different basic sets.

Let us check that there are no $\mathcal{A}$-subgroups of order $3$ distinct from $C_1$. Assume the contrary. Without loss of generality let $S$ be an $\mathcal{A}$-subgroup of order $3$ other than $C_1$. Let $\pi_1$  be the quotient epimorphism from $D$ to $D/S$. The set $\pi_1(X)$ has trivial radical since otherwise $\rad(\pi_1(X))=C_1$ and, hence, $\rad(X)=\langle sc_1 \rangle$ or $\rad(X)=\langle s^2c_1 \rangle$. Thus  we have $\pi_1(X)\in Orb(K,C)$ for some $K\leq \aut(C)$ by Theorem \ref{cyclering}. Lemma \ref{cycleorbit} implies that $\pi_1(X)=\{u\}$ or $\pi_1(X)=\{u,u^{-1}\}$. The first case is impossible because in this case $S\leq \rad(X)$. In the second case $y=x^k=x^{-1}$ or $z=x^l=x^{-1}$. This contradicts the fact that $k\equiv 1 \mod 3$ and $l\equiv 1 \mod 3$. Thus there are no $\mathcal{A}$-subgroups of order $3$ distinct from $C_1$ and $(9)$ yields that $\mathcal{S}(\mathcal{A}_E)$ has (up to Cayley isomorphism) form $(1)$ or $(3)$ from Lemma \ref{sringE}. 

Let $Y$ and $Z$ be the basic sets containing the elements $y$ and $z$ respectively. Then by the above remarks,  $Y^{-1}=\{y^{-1}\}\cup s\{u\}\cup s^2\{v\},~u,v\in C$, and $Z^{-1}=\{z^{-1}\}\cup s\{u_1\}\cup s^2\{v_1\},~u_1,v_1\in C$. The element $s$  enters the element $\underline{X}~\underline{Y}^{-1}$. So the elements $sc_1$ and $sc_1^2$ enter $\underline{X}~\underline{Y}^{-1}$. Therefore without loss of generality $u=c_1x^{-1},~v=c_1^2z^{-1}$. The element $s^2$  enters the element $\underline{X}~\underline{Z}^{-1}$. So the elements $s^2c_1$ and $s^2c_1^2$ enter $\underline{X}~\underline{Z}^{-1}$. Therefore without loss of generality $u_1=c_1y^{-1},~v_1=c_1^2x^{-1}$. Thus 
$$Y=\{y\}\cup s\{c_1z\}\cup s^2\{c_1^2x\},~Z=\{z\}\cup s\{c_1x\}\cup s^2\{c_1^2y\}.$$ 
By Theorem \ref{burn}, we have  $Y=X^{(k)},~Z=X^{(l)}$. Since $k\equiv 1 \mod 3$ and $l\equiv 1 \mod 3$, we conclude that $y^k=c_1z,~z^k=c_1^2x,~y^l=c_1x,~z^l=c_1^2y$. These equlities imply that $x^{k^3}=y^{k^2}=x,~x^{l^3}=z^{l^2}=x$. Thus $3^n$ divides  $k^3-1$ and $l^3-1$.  Let $k=3^ap+1$, where $3$ does not divide $p$ and $a\geq 1$. Then 
$$k^3-1=3^{a+1}p(3^{2a-1}p^2+3^ap+1).$$ 
We have $a+1\geq n$ because $3$ does not divide $(3^{2a-1}p^2+3^ap+1)$. Hence $y=x^k=q_1x$ where $q_1\in C_1$. Similarly $z=q_2x,~q_2\in C_1$. Since $\rad(X)=e$, we have $q_2\neq q_1^{-1}$. Every highest basic set is rationally conjugate to $X$ because $X_{0e}\neq \varnothing,~X_{1s}\neq \varnothing,~X_{2s^2}\neq \varnothing$. This implies that if $X$ is an orbit of $K\leq \aut(D)$ then every highest basic set is an orbit of $K$.

A straightforward computation shows that  $q_1q_2x^3$ is the unique element which enters the  element $\underline{X}~\underline{X}^{(3)}$ with coefficient $3$. So the set $\{x^3\}$ is basic. If $\mathcal{S}(\mathcal{A}_E)$ is of the form $(3)$ from Lemma \ref{sringE} then  the basic sets of $\mathcal{A}_{D_{n-1}}$ are of the forms 
$$\{u\},~suC_1,~s^2uC_1,~u\in C_{n-1}$$ 
by Lemma \ref{aset}, and, hence, $\mathcal{A}$ is Cayley isomorphic to $Cyc(K_6,D)$; if $\mathcal{S}(\mathcal{A}_E)$ is of the form $(1)$ from Lemma \ref{sringE} then  the basic sets of $\mathcal{A}_{D_{n-1}}$ are of the forms 
$$\{u\},~suC_1\cup s^2uC_1,~u\in C_{n-1}$$ 
by Lemma \ref{aset}, and, hence, $\mathcal{A}$ is Cayley isomorphic to $Cyc(K_7,D)$.

\textbf{Case 5: $\boldsymbol{|X|=6}$}. From Lemma \ref{lowset} it follows that 
$$X=\{x,x^{-1}\}\cup s\{y,z\}\cup s^2\{y^{-1},z^{-1}\},~x,y,z\in C.$$ 
Theorem \ref{burn} implies that for every highest basic $Y$  there exists $m\in \mathbb{Z}$ such that $Y=X^{(m)}$. Since $C$ is cyclic, there exist $k,l\in \mathbb{Z}$ such that $y=x^k,z=x^l$. Without loss of generality we may assume that 
$$k\equiv 1 \mod 3,~l\equiv 2 \mod 3.$$
Indeed, suppose that $k \equiv l \mod 3$. Then the element $s^2yz$ has order $3^n$ and it enters the element $\theta_2=\underline{X}^2-\underline{X}^{(2)}$. So there exists $b\in C,~|b|=3^n$ that enters $\theta_2$ because every highest basic set nontrivially intersects with $C$. However, only the elements $yz^{-1},~y^{-1}z,~e$ from $C$ enter $\theta_2$. All these elements have order less than $3^n$, a contradiction. Therefore $k \not\equiv l \mod 3$.

Let us prove $(9)$. The proof  is similar to the case when $|X|=3$. If $|xE\cap X|=3$ then 
$$X=x\{e,s\varepsilon_1,s^2\varepsilon_2\}\cup x^{-1}\{e,s^2\varepsilon_1^2,s\varepsilon_2^2\},~\varepsilon_1,\varepsilon_2\in C_1,$$
because $X=X^{-1}$. Note that $\varepsilon_2\neq \varepsilon_1^{-1}$ since otherwise $\rad(X)>e$. Thus all elements from $E\setminus C_1$ enter $\underline{X}^2$ and only these elements from $D_{n-1}$ enter $\underline{X}^2$. Since $\mathcal{A}$ is regular, $D_{n-1}$ is an $\mathcal{A}$-subgroup. So $E\setminus C_1$  is an $\mathcal{A}$-set. Lemma \ref{radbasic} implies that $E=\langle E\setminus C_1 \rangle$ is an $\mathcal{A}$-subgroup and we are done. 

If $|xE\cap X|<3$ then from Theorem \ref{sch} it follows that $X^{[3]}=\{x^3,y^3,z^3,x^{-3},y^{-3},z^{-3}\}\subseteq C$ is an $\mathcal{A}$-set. All basic sets inside $X^{[3]}$ are conjugate  and, hence, have the same radical. Suppose that every basic set inside $X^{[3]}$ has trivial radical. Then from Lemma \ref{cyclering} and Lemma \ref{cycleorbit} it follows that $T_{x^3}=\{x^3\}$ or $T_{x^3}=\{x^3,x^{-3}\}$. In the former case $x^3X^{(2)}$ is a basic set by Lemma \ref{aset}. So $x^3X^{(2)}=X$ and $x^5=x$, a contradiction with $|x|>3$. In the latter case $(x^3+x^{-3})\underline{X}^{(2)}$ contains $12$ elements including $x$ and $x^5$. This implies that 
$$(x^3+x^{-3})\underline{X}^{(2)}=\underline{X}+\underline{X}^{(5)}.~\eqno(10)$$
A straightforward computation shows that $s\{x^3y^{-2},x^3z^{-2},x^{-3}y^{-2},x^{-3}z^{-2}\}$ is the set of all elements from $sC$ that the summand in the left-hand side of $(10)$ contains. Therefore 
$$\{x^3y^{-2},x^3z^{-2},x^{-3}y^{-2},x^{-3}z^{-2}\}=\{y,z,y^{-5},z^{-5}\}.$$ 
Since  $k\equiv 1 \mod 3,~l\equiv 2 \mod 3$, we conclude that $y^3,z^3\in\{x^3,x^{-3}\}$. This yields that $|xE\cap X|=3$, in contrast to the assumption. Thus $\rad(X^{[3]})>e$. Moreover, $\rad(X^{[3]})=C_1$ because $|X^{[3]}|\leq 6$. From Lemma \ref{radbasic} it follows that $C_1$ is $\mathcal{A}$-subgroup. If $T_s$ is not regular then like in the case when $|X|=3$ we have that $|T_s|\geq 18$. This contradicts the fact that exactly $16$ elements including $s$ from $sC$ enter $\underline{X}^2$. Thus $T_s$ is regular and $E$ is an $\mathcal{A}$-subgroup.

Now we show that $\{c_1,c_1^2\}\in \mathcal{S}(\mathcal{A})$. Assume the contrary. Then $\{c_1\}$ is a basic set. So $c_1X$ is a basic set. Without loss of generality we assume that $x^{3^{n-1}}=c_1$. Since $k\equiv 1 \mod 3$ and $l\equiv 2 \mod 3$, we  conclude that $y^{3^{n-1}}=c_1$ and $z^{3^{n-1}}=c_1^2$.  On the one hand, 
$$X^{(3^{n-1}+1)}=\{c_1x,c_1^2x^{-1}\}\cup s\{c_1y,c_1^2z\}\cup s^2\{c_1^2y^{-1},c_1z^{-1}\}$$
 is a basic set containing $c_1x$. On the other hand, $X^{(3^{n-1}+1)}\neq c_1X$ because $sc_1z\in c_1X,~sc_1z\notin X^{(3^{n-1}+1)}$, a contradiction. Thus $\{c_1,c_1^2\}\in \mathcal{S}(\mathcal{A})$.

Let us show that there are no $\mathcal{A}$-subgroups of order $3$ distinct from $C_1$. Assume the contrary. Without loss of generality let $S$ be an $\mathcal{A}$-subgroup of order $3$ other than $C_1$. Let $\pi_1$  be the quotient epimorphism from $D$ to $D/S$. Suppose that $\rad(\pi_1(X))>e$. Then $\rad(\pi_1(X))=C_1$ because $|\pi_1(X)|\leq 6$. This implies that $\pi_1(X)=xC_1\cup x^{-1}C_1$. So 
$$X=x\{e,f_1,f_2\}\cup x^{-1}\{e,f_1^{-1},f_2^{-1}\},$$
where $f_1,f_2\in \{sc_1,sc_1^2,s^2c_1,s^2c_1^2\}$. Assume that $f_2\neq f_1^{-1}$. Then $f_2=ff_1$, where $f\in\{s,s^2,c_1,c_1^2\}$. Note that $\{f,f^2\}$ is an $\mathcal{A}$-set. The element $f_2x$ enters $(f+f^2)\underline{X}$   whereas $x$ does not. A contradiction holds because $x$ and $f_2x$ lie in the same basic set $X$. Therefore $f_2= f_1^{-1}$ and $\rad(X)=\{e,f_1,f_1^{-1}\}$ that is not true. Thus $\rad(\pi_1(X))=e$. Hence Lemma \ref{cyclering} and Lemma \ref{cycleorbit} yield that $\pi_1(X)=\{u,u^{-1}\}$. Then $S\leq \rad(X)$, a contradiction. So there are no $\mathcal{A}$-subgroups of order $3$ distinct from $C_1$. Thus by the previous paragraph, $\mathcal{S}(\mathcal{A}_E)$ has form $(6)$ or $(9)$ from Lemma \ref{sringE}.

Further we prove that $y=q_1x,~z=q_2x^{-1},~q_1,q_2\in C_1$. Using the assumption on $l$ and $k$, list all elements from $sC$ of order less than $3^n$ that enter $\underline{X}^{-1}~\underline{X}^{(k)}$ 

$$s,~sz^kx,~sy^kx^{-1},~syz,~sz^{-k}y^{-1},~sy^{-k}z^{-1}.~\eqno(11)$$
and all elements from $sC$ of order less than $3^n$ that enter $\underline{X}^{-1}~\underline{X}^{(l)}$
$$s,~sx^{-1}y^{-l},~syz,~sz^{l}y^{-1},~sy^lz^{-1},~sxz^{-l}.~\eqno(12)$$
Every element from $(11)$ has the form $sx^i$, where 
$$i\in I=\{kl+1,~k^2-1,~k+l,~kl+k,~k^2+l\},$$ 
because $y=x^k,~z=x^l$. Since $c_1\in \rad(T_s)$ and $T_s$ is contained in $\underline{X}^{-1}~\underline{X}^{(k)}$, we conclude that $sc_1$ and $sc_1^2$ are among the elements in $(11)$. 
Therefore  two numbers from $I$ are divisible by $3^{n-1}$. Let us show that   $3^{n-1}$ divides $k-1$. This is obvious by the assumption on $k$ if $k^2-1$ is a multiple of $3^{n-1}$. If $k+l$ and $kl+k$ are divisible by  $3^{n-1}$ then $3^{n-1}$ divides $l(k-1)=kl+k-k-l$, hence, $3^{n-1}$ divides $k-1$ by the assumption on $l$. If $k+l$ and $kl+1$ are multiple of $3^{n-1}$ then $(k-1)(l-1)=kl+1-k-l$ is a multiple of $3^{n-1}$, hence,  $k-1$ is a multiple of $3^{n-1}$. If $k+l$ and $k^2+l$ are divisible by $3^{n-1}$ then $k(k-1)=k^2+l-l-k$ is divisible by $3^{n-1}$, hence,  $k-1$ is divisible by  $3^{n-1}$. If $3^{n-1}$ divides $kl+k$ and $kl+1$ then $3^{n-1}$ divides $k-1=kl+k-kl-l$. If $kl+k$ and $k^2+l$ are multiple of $3^{n-1}$ then $(k-1)(k-l)=k^2+l-kl-k$ is a multiple of $3^{n-1}$, hence, $k-1$ is a multiple of  $3^{n-1}$. If $kl+1$ and $k^2+l$ are divisible by $3^{n-1}$ then without loss of generality $z^kx=x^{kl+1}=c_1,~y^{-k}z^{-1}=x^{-k^2-l}=c_1^2$. Note that $x^{-k^3+1}=(x^{-k^2-l})^kx^{kl+1}=c_1^2c_1=e$. Thus $k^3-1$ is divisible by $3^n$ and $k-1$ is divisible by $3^{n-1}$. So in all cases $k-1$ is divisible by $3^{n-1}$ and we have $y=q_1x,~q_1\in C_1$. 

Every element from $(12)$ has the form $sx^j$, where 
$$j\in J=\{kl+1,~l^2-1,~k+l,~kl-l,~l^2-k\},$$ 
because $y=x^k,~z=x^l$. Again $sc_1$ and $sc_1^2$ are among the elements in $(12)$ because $T_s$ enters $\underline{X}^{-1}~\underline{X}^{(l)}$. This implies that $3^{n-1}$ divides  two elements from $J$. Applying the similar arguments, we conclude that $l+1$ is divisible by $3^{n-1}$ and $z=x^{l+1}x^{-1}=q_2x^{-1}$, where $q_2\in C_1$. Note $q_1\neq q_2$ since otherwise $sq_1\in \rad(X)$. Every highest basic set is rationally conjugate to $X$ because $X_{0e}\neq \varnothing,~X_{1s}\neq \varnothing,~X_{2s^2}\neq \varnothing$. This implies that if $X$ is an orbit of $K\leq \aut(D)$ then every highest basic set is an orbit of $K$.

A straightforward computation shows that only the elements $q_1q_2^2x^3,~q_1^2q_2x^{-3}$ enter the element $\underline{X}^{(3)}$ with coefficient $3$. If $\{q_1q_2^2x^3\}$ is a basic set then Lemma \ref{aset} implies that $\{c_1\}$ is a basic set that is not true. Therefore $\{q_1q_2^2x^3,~q_1^2q_2x^{-3}\}$ is a basic set and every basic set of $\mathcal{A}_{C_{n-1}}$ is of the form $\{u,u^{-1}\},~u\in C_{n-1},$ by Theorem \ref{burn} and Theorem \ref{sch}.  If $\mathcal{S}(\mathcal{A}_E)$  is of the form $(6)$ from Lemma \ref{sringE} then a straightforward check shows that every nonhighest basic set outside $C$ is  one of two forms 
$$\{su,sc_1u,sc_1^2u,su^{-1},sc_1u^{-1},sc_1^2u^{-1}\},~\{s^2u,s^2c_1u,s^2c_1^2u,s^2u^{-1},s^2c_1u^{-1},s^2c_1^2u^{-1}\},u\in C.$$ 
So $\mathcal{A}$ is Cayley isomorphic to $Cyc(K_8,D)$. If $\mathcal{S}(\mathcal{A}_E)$ is of the form $(9)$ from Lemma~\ref{sringE} then a straightforward check shows that every nonhighest basic set outside $C$ is of the  form 
$$\{su,sc_1u,sc_1^2u,s^2u^{-1},s^2c_1u^{-1},s^2c_1^2u^{-1}\},~u\in C.$$ 
So $\mathcal{A}$ is Cayley isomorphic to $Cyc(K_9,D)$. 
\end{proof}

\begin{corl}\label{regc1}
Let $\mathcal{A}$ be a regular $S$-ring over $D$ such that $\rad(\mathcal{A})=e$. Then $C_1$ and $C_{n-1}$ are  $\mathcal{A}$-subgroups.

\end{corl}

\begin{corl}\label{regisol}
Let $\mathcal{A}$ be a regular $S$-ring over $D$ such that $\rad(\mathcal{A})=e$. If $L$ is an  $\mathcal{A}$-subgroup of order $3$ and  $\mathcal{A}=Cyc(K_i,D),~i\in \{0,1,2,3,4,5\}$, then the group  $\aut(\mathcal{A}_{D/L})$ is $2$-isolated. 
\end{corl}
\begin{proof}
For $n\leq 3$ the statement of the corollary follows from  calculations in the group ring of $D$ that made by \cite{GAP}. So we assume that $n\geq 4$. By Lemma \ref{regorb}, it is sufficient to prove that $\aut(\mathcal{A}_{D/L})_e$ has a faithful regular orbit. If $L\neq C_1$ then $D/L$ is cyclic and  Theorem \ref{regular} implies that either $\mathcal{A}_{D/L}=\mathbb{Z}(D/L)$ or every basic set of $\mathcal{A}_{D/L}$ is of the form $\{x,x^{-1}\},~x\in C/L$. In both cases, obviously, the group $\aut(\mathcal{A}_{D/L})_e$ has a faithful regular orbit.

Let $L=C_1$ and $\pi:D\rightarrow D/L$ be the quotient epimorphism. If $\mathcal{A}=Cyc(K_0,D)$ then $|\aut(\mathcal{A}_{D/L})_e|=1$ and, obviously, $\aut(\mathcal{A}_{D/L})_e$ has a faithful regular orbit. If  $\mathcal{A}=Cyc(K_1,D)=\mathbb{Z}C\otimes \mathcal{A}_S$, where $\mathcal{A}_S$ has rank $2$,  then $\mathcal{A}_{D/L}=\mathbb{Z}C_{n-1}\otimes \mathcal{A}_S$. So $|\aut(\mathcal{A}_{D/L})_e|=2$ and $\pi(\{sx,s^2x\})$ is a faithful regular orbit of $\aut(\mathcal{A}_{D/L})_e$ . If $\mathcal{A}=Cyc(K_i,D),~i\in \{2,4,5\}$,   then $\mathcal{A}_{D/L}$ is a quasi-thin $S$-ring  with at least two orthogonals. Indeed, $\{\pi(y^2),\pi(y^{-2})\}$ and $\{\pi(y^4),\pi(y^{-4})\}$, where $y$ is a generator of $C_{n-1}$, are orthogonals. These orthogonals are distinct since $n\geq 4$. Thus $\aut(\mathcal{A}_{U/L})_e$ has a faithful regular orbit by Lemma \ref{quasithin}. If $\mathcal{A}=Cyc(K_3,D)$ the basic sets of $\mathcal{A}_{D/L}$ are of the form 

$$\{\pi(s),\pi(s^2)\},~\{\pi(y),\pi(y^{-1})\},~\pi(s)\{\pi(y),\pi(y^{-1})\}\cup \pi(s^2)\{\pi(y),\pi(y^{-1})\},~y\in C.$$

Note that $\pi(S)$ and $\pi(C)$ are $\mathcal{A}_{D/L}$-subgroups, $\mathcal{A}_{D/L}=\mathcal{A}_{C/L}\otimes\mathcal{A}_{S/L}$, and $\aut(\mathcal{A}_{D/L})_e|_S\cong \aut(\mathcal{A}_{D/L})_e|_C\cong \mathbb{Z}_2$. Thus we conclude that $\aut(\mathcal{A}_{D/L})_e\cong \mathbb{Z}_2\times \mathbb{Z}_2 $ and $\pi(s)\{\pi(y),\pi(y^{-1})\}\cup \pi(s^2)\{\pi(y),\pi(y^{-1})\}$ is a faithful regular orbit of $\aut(\mathcal{A}_{D/L})_e$.
\end{proof}

\begin{prop}\label{trivrad}
Let $\mathcal{A}$ be an $S$-ring over $D$. Suppose that for every  highest basic set $X$ with trivial radical the following statements hold:

$(1)$ $X$ is regular;

$(2)$  $\langle X \rangle=D$. 

Then every highest basic set of $\mathcal{A}$ has trivial radical. 
\end{prop}

\begin{proof}
Lemma \ref{lowset} implies that that $|X|\leq 6$. If $|X|=1$ then Condition $(2)$ of the Proposition does not hold. The statement of the proposition follows from Theorem \ref{burn} if $|X|\in \{3,6\}$. In these cases every highest basic set is rationally conjugate to $X$ because  $X_{0e}\neq \varnothing,~X_{1s}\neq \varnothing,~X_{2s^2}\neq \varnothing$. Suppose that $|X|=2$. Then without loss of generality we may assume that $X=\{x\}\cup s\{y\},~x,y\in C$. Theorem~\ref{sch} implies that $C_1$ is an  $\mathcal{A}$-subgroup. From $(6)$ (see proof of Theorem \ref{regular}) it follows that $T_s$ is regular (otherwise $|T_s|\geq 6$ by Lemma \ref{nonregc1}), $|T_s|\leq 2$, and $T_s\neq T_{s^2}$. Thus $E$ is an  $\mathcal{A}$-subgroup and Lemma \ref{sringE} yields that there exists a basic set of the form $\{q\},~q\in E\setminus C_1$. From Lemma \ref{aset} it follows that $qX$ is a basic set. Therefore 
$$\bigcup\limits_{m}{(X^{(m)}\cup (qX)^{(m)})}=D\setminus D_{n-1},$$
where $m$ runs over integers coprime to $3$. We are done because $\rad(X)=e$. 

Let us show that Condition $(2)$ of the Proposition does not hold whenever  $|X|=4$. In this case without loss of generality we may assume that $X=\{x,x^{-1}\}\cup s\{y\}\cup s^2\{y^{-1}\},~x,y\in C$. Theorem \ref{sch} implies that $C_1$ is an  $\mathcal{A}$-subgroup. Exactly one of the elements $sxy,sx^{-1}y$, say the first one, has order $3^n$. If $T_{sxy}$ is not regular then $|T_{sxy}|\geq 6$ by Lemma \ref{nonregc1} that contradicts  $(8)$ (see proof of Theorem~\ref{regular}). Indeed, exactly four distinct elements including $sxy$ enter $\underline{X}^2-\underline{X}^{(2)}$. So $T_{sxy}$ is regular and from $(8)$ it follows that  $T_{sxy}=\{sxy\}$ or $T_{sxy}=\{sxy,s^2x^{-1}y^{-1}\}$. In both cases $\langle T_{sxy} \rangle$ is cyclic.
\end{proof}

\section{$S$-rings over $D=\mathbb{Z}_3\times \mathbb{Z}_{3^n}$: nontrivial radical case }

In this section we are interested in an $S$-ring over $D$ that has nontrivial radical and follow \cite[Section~9]{MP3}.

\begin{theo}\label{nontrivrad}
Let $\mathcal{A}$ be an $S$-ring over  $D$ such that $\rad(\mathcal{A})>e$. Then $\mathcal{A}$ is a proper $S$-wreath product. Moreover, $|S|=1$ or $|S|=3$ or $\rad(\mathcal{A}_U)=e$ and $|L|=3$, where $S=U/L$.
\end{theo}
\begin{proof}

Denote by $U$ the group generated by all basic sets of $\mathcal{A}$ with trivial radicals.

\begin{lemm} \label{Urad}
$U$ is an $\mathcal{A}$-subgroup and $\rad(\mathcal{A}_{U})=e$.
\end{lemm}

\begin{proof}

The first statement follows from Lemma \ref{radbasic}. Let us prove that $\rad(\mathcal{A}_{U})=e$. Without loss of generality we may assume that  $U=D$. Then there exists a highest basic set $X$ such that $\rad(X)=e$. If $X$ is not regular then  it contains an element of order $3$ by Lemma \ref{nonregset} and  we have that $\rad(\mathcal{A})=e$ by Proposition \ref{nonreghigh}. Therefore we may assume that every highest basic set with trivial radical is regular. At least one of the sets $X_{0e},X_{1s},X_{2s^2}$ is not empty. Without loss of generality  we  assume that $X_{0e}\neq \varnothing$. If $\langle Z \rangle=D$ for every highest basic set $Z$ with trivial radical  then  every highest basic set has trivial radical by Proposition \ref{trivrad} and we are done. So we may assume that $\langle X \rangle=C$. Then by Lemma \ref{cycleorbit}, we have $X=\{x\}$ or $X=\{x,x^{-1}\}$, where $x\in C$. Since $\langle X \rangle=C$, there exists a basic set $Y$ such that $D=\langle X,Y \rangle$ and $\rad(Y)=e$. If $Y$ is not regular then from Lemma \ref{nonregc1} it follows that $C_1\leq \rad(Y)$. So $Y$ is  regular. Without loss of generality $Y_{1s}\neq \varnothing$ and $y\in Y_{1s}$. Note that $|Y_{1s}|\leq 2$ by Lemma \ref{lowset}.

Let us show that $\rad(T_s)=e$, where $T_s$ is a basic set containing $s$. It is obvious if $T_s=Y$. Suppose that $T_s\neq Y$. Assume on the contrary that $\rad(T_s)>e$. Then $C_1\leq \rad(T_s)$. If $|Y_{1s}|=1$ then exactly two elements $s$ and $sy^2$ from $sC$ enter the element $\alpha=(y+y^{-1})\underline{Y}$, a contradiction. Suppose that $Y_{1s}=\{y,z\}$. Then exactly four elements $s$, $syz$, $sy^{-1}z$, and $sy^2$ from $sC$ enter $\alpha$. Since $sc_1,sc_1^2\in T_s$, we  conclude that  $\{c_1,c_1^2\}\subseteq\{y^2,yz,y^{-1}z\}$. If $y^2\in C_1$ then $sc_1\in Y$ or $sc_1^2\in Y$ and $Y=T_s$, a contradiction. Thus $\{c_1,c_1^2\}=\{yz,y^{-1}z\}$. Hence $z^2=e$, a contradiction.  

From the claim of the previous paragraph it follows that $T_s$ is regular because otherwise $C_1\leq \rad(T_s)$ by Lemma \ref{nonregc1}. Since $C_1$ is an $\mathcal{A}$-subgroup, from Lemma \ref{sringE} it follows that there exisits an $\mathcal{A}$-subgroup $Q\neq C_1$ of order $3$. Every basic set inside $C$ has the form $\{x\}$ or  $\{x,x^{-1}\}$. Hence  $\{qx,qx^{-1},q^2x,q^2x^{-1}\}$ is an $\mathcal{A}$-set for every $x\in C$. Thus every  basic set of $\mathcal{A}$ has  trivial radical as required.
\end{proof}

\begin {lemm} \label{minagr}
Let $\mathcal{A}$be an $S$-ring over  $D$ and $H$ be a minimal $\mathcal{A}$-subgroup. Then either $H \setminus \{e\}\in \mathcal{S}(\mathcal{A})$ or $|H|=3$ and $\mathcal{A}_H=\mathbb{Z}H$.
\end{lemm}

\begin{proof}
If $|H|>3$ then the primitive $S$-ring $\mathcal{A}_H$ is of rank $2$ by Lemma \ref{bgroups}. Thus $H\setminus \{e\} \in \mathcal{S}(\mathcal{A})$. Let  $|H|=3$ and $rk(\mathcal{A}_H)>2$. Then, obviously, $\mathcal{A}_H=\mathbb{Z}H$.
\end{proof}

Since by the theorem hypothesis $\rad(\mathcal{A})>e$, Lemma \ref{Urad} implies that $U<D$. In addition, from Lemma \ref {minagr} it follows that $U$ contains every minimal $\mathcal{A}$-subgroup.
 
\begin{lemm}\label{uniqmin}
If there is a unique minimal $\mathcal{A}$-subgroup or $C_1\leq \rad(X)$ for every basic set $X$ outside $U$ then the statement of Theorem \ref{nontrivrad} holds.

\end{lemm}

\begin{proof}
Let $L$ be a unique minimal $\mathcal{A}$-subgroup. Then  $L\le\rad(X)$ for every $X\in \mathcal{S}(\mathcal{A})_{D\setminus U}$  because  Lemma \ref{radbasic} implies that  $\rad(X)$ is a nontrivial $\mathcal{A}$-subgroup. So $\mathcal{A}$ is the $U/L$-wreath product. If $|L|=3$ then the statement of  Theorem \ref{nontrivrad} holds.  Suppose that $|L|>3$. Then from Corollary \ref{regc1} it follows that   $\mathcal{A}_U$ is not regular (otherwise $C_1$ is an $\mathcal{A}$-subgroup of order $3$). Therefore by Theorem \ref{nonreg}, the $S$-ring $\mathcal{A}_U$ has rank $2$. Hence $U=L$ and we are done.

To complete the proof suppose that there are at least two minimal $\mathcal{A}$-subgroups and $C_1 \leq \rad(X)$ for every $X\in \mathcal{S}(\mathcal{A})_{D\setminus U}$. Assume that $C_1\nleq U$. Then $|U|=3$ and $U$ is the unique minimal $\mathcal{A}$-subgroup becuse $U$ contains every minimal $\mathcal{A}$-subgroup, a contradiction. So $C_1\leq U$. Denote by $H$ the minimal $\mathcal{A}$-subgroup containing $C_1$. Then $\mathcal{A}$ is the $S$-wreath product, where $S=U/H$. If $H=C_1$ then $|H|=3$ and the statement of Theorem \ref{nontrivrad} holds. If $H>C_1$ then $\mathcal{A}_U$ is not regular by Corollary \ref{regc1}. Thus Theorem \ref{nonreg} implies that $\mathcal{A}_U=\mathcal{A}_H\otimes \mathcal{A}_L$ where $|L|\leq 3$. So $|S|=|U/H|\in \{1,3\}$.
\end{proof}

The uninon of all basic sets $X$ such that $\rad(X)=e$ or $c_1\in \rad(X)$ denote by $V$. Then, obviously, $U\subset V$ and  $V$ is an $\mathcal{A}$-set. By  Lemma \ref{uniqmin}, we may assume that $V\neq D$ and there exist at least two minimal $\mathcal{A}$-subgroups. Let $X$ be a basic set containing an element of order $3$. Suppose that $\rad(X)> e$ and $c_1\notin \rad(X)$. Then $|\rad(X)|=3$ and $\rad(X)$ is the unique minimal $\mathcal{A}$-subgroup, a contradiction. So $X\subseteq V$ and $E\subseteq V$.

For given $\mathcal{A}$-set $X$ put $\mathcal{S}(\mathcal{A})_{X}=\{Y\in\mathcal{S}(\mathcal{A}):Y\subseteq X\}$.
\begin{lemm}\label{radv}
Let $X\in \mathcal{S}(\mathcal{A})_{D\setminus V}$. Then
\begin{enumerate}
\item $\rad(X)=\{e,q,q^2\},~q\in E\setminus C_1$;
\item $X$ regular and $X_{0e}\neq \varnothing,~X_{1s}\neq \varnothing,~X_{2s^2}\neq \varnothing.$
\end{enumerate}
\end{lemm}

\begin{proof}
Since $X \not \subseteq U$, we have $\rad(X)>e$. The group $\rad(X)$ does not contain elements of order greater than $3$ because by the above assumption $c_1\notin \rad(X)$. Therefore $\rad(X)$ is a group of order $3$ distinct from $C_1$ and Statement $(1)$ holds. Denote  $\rad(X)=\{e,q,q^2\}$ by $L$. Then $\rad(\pi(X))=e$, where $\pi:D\rightarrow D/L$ is  the quotient epimorphism. The group $D/L$ is cyclic. Thus by Lemma \ref{cyclering}, the set $\pi(X)$  is regular or $\pi(X)=H\setminus\{e\}$ for some $H\leq D/L$. If $\pi(X)=H\setminus\{e\}$  then $X$ contains all elements of order $3$ distinct from $q$ and $q^2$. Hence $L$ is a unique minimal $\mathcal{A}$-subgroup, a contradiction. So  $\pi(X)$ and $X$ are regular. Since $LX=X$, we  conclude that $X_{0e}\neq \varnothing,~X_{1s}\neq \varnothing,~X_{2s^2}\neq \varnothing$.
\end{proof}

 Statement $(2)$ of Lemma \ref{radv} yields that 
$$\bigcup\limits_{X\in\mathcal{S}(\mathcal{A})_{D\setminus V}}tr(X)=D\setminus D_k$$
for some $k\geq 1$. So  $V=D_k$ is an $\mathcal{A}$-subgroup.

\begin{lemm}\label{samerad}
Let $X,~Y\in \mathcal{S}(\mathcal{A})_{D\setminus V}$. Then $\rad(X)=\rad(Y)$.
\end{lemm}

\begin{proof}
Assume the contrary. Then by Lemma  \ref{radv}, without loss of generality we may assume that  $\rad(X)=\{e,s,s^2\}=S$ and $\rad(Y)=\{e,sc_1,sc_1^2\}.$ So  
$$X=X_{0e}\cup sX_{0e}\cup s^2X_{0e},~Y=Y_{0e}\cup sc_1Y_{0e}\cup s^2c_1^2Y_{0e}.$$ 
The sets $X_{0e}$ and $Y_{0e}$ are regular and nonempty by Statement $(2)$ of Lemma \ref{radv}. Thus $X_{0e}$ and $Y_{0e}$ are orbits of a group $K\leq \aut(C)$. Furthermore, $\rad(X_{0e})=\rad(Y_{0e})=e$ since otherwise $\rad(X)$ or $\rad(Y)$ contains at least nine elements that contradicts Statement $(1)$ of Lemma \ref{radv}. Consider the basic sets   $\pi(X)$ and $\pi(Y)$ of the circulant $S$-ring $\mathcal{A}_{D/S}$, where  $\pi:D\rightarrow D/S$ is the quotient epimorphism. The radical of $\pi(X)=X_{0e}$  is trivial, the radical $\pi(Y)$ contains $\pi(sc_1)$. This contradicts Lemma \ref{cyclesection} applied  for an $S$-ring  $\pi(\mathcal{A}_{\langle X \rangle})$ and $\mathcal{A}$-section $\pi(\langle Y \rangle)$.
\end{proof}

\begin{lemm}
$\rad(\mathcal{A}_V)=e.$ In particular, $U=V$.
\end{lemm}

\begin{proof}
Let $X$ be the highest basic set of $\mathcal{A}_V$. Since $V\neq D$, there exists $Y\in  \mathcal{S}(\mathcal{A})_{D\setminus V}$ such that $X\cap Y^3\neq \varnothing$. Moreover, $X\subset Y^3$ because $Y^3$ is an $\mathcal{A}$-set. By Lemma \ref{radv}, without loss of generality  we may assume that $Y=Y_{0e}\cup sY_{0e}\cup s^2Y_{0e},~\rad(Y)=S,~\rad(Y_{0e})=e$. Then $Y^3=SY_{0e}^3$. Since $Y_{0e}$ is nonempty and regular by Lemma \ref{radv}, the set $Y_{0e}$ is an orbit of $K\leq \aut(C)$. Therefore Lemma \ref{cycleorbit} implies that either $Y_{0e}=\{y\}$ or $Y_{0e}=\{y,y^{-1}\}$.  So $Y^3$ is one of two types:
$$\{y^3,sy^3,s^2y^3\},$$
$$\{y,sy,s^2y,y^{-1},sy^{-1},s^2y^{-1},y^3,sy^3,s^2y^3,y^{-3},sy^{-3},s^2y^{-3}\}.$$
Since $X\subset V$, we have $\rad(X)=e$ or $c_1\in \rad(X)$. Let us show that the latter is impossible. This is obvious  if $Y^3$ is of the first type. If $Y^3$ is of the second type and $c_1\in \rad(X)$ then it is easy to check that $y=c_1$ or $y=c_1^2$. However, $y\notin V$. This contradicts the fact that  $E\subseteq V$. Thus $\rad(X)=e$ and $\rad(\mathcal{A}_V)=e.$
\end{proof}

To complete the proof of Theorem \ref{nontrivrad}, let $L=\rad(X),~X\in \mathcal{S}(\mathcal{A})_{D\setminus V}$. By Lemma \ref{radv}, the group $L$ does not depend on the choice of $X$. Then   $\mathcal{A}$ is the $S$-wreath product where $S=V/L$. Since $\rad(\mathcal{A}_V)=e$ and $|L|=3$ the statement of Theorem \ref{nontrivrad} holds.
\end{proof}

\section{$S$-rings over $D=\mathbb{Z}_3\times \mathbb{Z}_{3^n}$: schurity }
In this section we prove Theorem \ref{main}.

\begin{proof}[Proof of the Theorem \ref{main}]
We proceed by induction on $n$. The statement of the theorem for $n\leq 3$ follows from   calculations in the group ring of $D$ that made by \cite{GAP}. Let $n\geq 4$ and $\mathcal{A}$ be an $S$-ring over $D$. Let us show that $\mathcal{A}$ is schurian. If $\rad(\mathcal{A})=e$ then $\mathcal{A}$ is schurian by Theorem \ref{nonreg} and Theorem \ref{regular}. Suppose that $\rad(\mathcal{A})>e$. Then Theorem \ref{nontrivrad} implies that $\mathcal{A}$ is the $S$-wreath product, where $S=U/L$, and one of the following statements holds: 

$(1)$ $|S|=1,$ 

$(2)$ $|S|=3,$ 

$(3)$ $\rad(\mathcal{A}_U)=e$ and $|L|=3.$

Let us show that in these cases $\mathcal{A}$ is schurian. By the induction hypothesis, the $S$-rings $\mathcal{A}_U$ and $\mathcal{A}_{D/L}$ are schurian. So by Corollary \ref{isolated},  it is sufficient to prove that  $\aut(\mathcal{A}_S)$ is $2$-isolated. Obviously, the group $\aut(\mathcal{A}_S)$ is $2$-isolated whenever  $|S|=1$ or $|S|=3$. Hence we may assume that $|S|\geq 9$. Let $\rad(\mathcal{A}_U)=e$ and $|L|=3$. If $U$ is cyclic then Lemma \ref{cyclering} and Lemma \ref{cycleorbit} imply that every basic set of $\mathcal{A}_U$ is of the form $\{x\}$ or every basic set of $\mathcal{A}_U$ is of the form $\{x,x^{-1}\}$. In these cases, obviously, $\aut{(\mathcal{A}_{U/L})}_e$ has a faithful regular orbit. So $\aut{(\mathcal{A}_{U/L})}$ is $2$-isolated by Lemma \ref{regorb}.  Thus we may assume that $U=D_k$, where $k<n$.  Since $L$ is an $\mathcal{A}$-subgroup of order $3$, we conclude that $rk(\mathcal{A}_U)>2$.

Suppose that $\mathcal{A}_U$ is not regular. Then Theorem \ref{nonreg} implies that $\mathcal{A}_U=\mathcal{A}_H\otimes \mathcal{A}_L$, where $rk(\mathcal{A}_H)=2$. 

If $rk(\mathcal{A}_L)=2$, we have 

$$\aut(\mathcal{A}_U)^S=(\sym(H)\times \sym(L))^{U/L}=\sym(U/L).$$ 

If $rk(\mathcal{A}_L)=3$, we have 

$$\aut(\mathcal{A}_U)^S=(\sym(H)\times L_{right})^{U/L}=\sym(U/L).$$ 
Note that $D/L$ is cyclic and $S$ is an $\mathcal{A}_{D/L}$-section of composite order.  From \cite[Theorem 4.6]{EP3} it follows that $\aut(\mathcal{A}_{D/L})^{S}=\sym(U/L)$. Thus  $\mathcal{A}$ is schurian by Theorem \ref{genwr} applied for $\Delta_0=\aut(\mathcal{A}_{D/L})$ and $\Delta_1=\aut(\mathcal{A}_U)$. Hence by Theorem \ref{nonreg}, we may assume that $\mathcal{A}_U$ is regular. From Theorem \ref{regular} it follows that $\mathcal{A}_U=Cyc(K,U)$, where $K$ listed in Table $1$. If $\mathcal{A}_U=Cyc(K_i,U),~i\in\{0,1,2,3,4,5\}$,  then the group $\mathcal{A}_{U/L}$ is $2$-isolated by Corollary \ref{regisol}. 

Let $\mathcal{A}_U=Cyc(K_i,U),~i\in\{6,7,8,9\}$. Then $C_1$ is the unique $\mathcal{A}$-subgroup of order $3$ and, hence, $L=C_1$. Besides, for every highest basic set $X$ of $\mathcal{A}_U$ we have $X_{0e}\neq \varnothing,~X_{1s}\neq \varnothing,~X_{2s^2}\neq \varnothing$. Let us show that  the group $H=\rad(\mathcal{A})$ is not cyclic. Assume the contrary. Note that $C_1\leq H<U$. The first inequality holds because otherwise $H$ is an $\mathcal{A}$-subgroup of order $3$ distinct from $C_1$. The latter inequality holds because otherwise $X\subseteq H$ for some highest basic set $X$ of $\mathcal{A}_U$ and $H$ can not be cyclic since $X_{0e}\neq \varnothing,~X_{1s}\neq \varnothing,~X_{2s^2}\neq \varnothing$. Next, the group $D/H$ is isomorphic to $D_l$, where $l<n$. The $S$-ring $\mathcal{A}_{D/H}$ has trivial radical because otherwise $H$ is not the radical of $\mathcal{A}$. The group $\widetilde{E}=\{x\in D/H: |x|=3\}$ is an $\mathcal{A}_{D/H}$-subgroup because every basic set of $\mathcal{A}_{U/H}$ is regular and $U/H$ contains all elements of order $3$. Theorem \ref{nonreg} applied to $D/H\cong D_l$ yields that $\mathcal{A}_{D/H}$ is regular. Then by Corollary \ref{regc1}, there exists a cyclic $\mathcal{A}_{D/H}$-subgroup of order $3^{l-1}$. This contradicts the fact that $X_{0e}\neq \varnothing,~X_{1s}\neq \varnothing,~X_{2s^2}\neq \varnothing$ for all highest basic sets of $\mathcal{A}_{U/H}$. Therefore $H$ is not cyclic and, hence, $D/H$ is cyclic.  

Since $H$ is not cyclic,  there exists a highest basic set $X$ of $\mathcal{A}$ such that $s\in \rad(X)$. By Theorem~\ref{burn}, every highest basic set of $\mathcal{A}$ is rationally conjugate to $X$ and $\rad(Y)=H$ for every highest basic set $Y$ of $\mathcal{A}$. Therefore $\mathcal{A}_{D/H}$ is circulant and $\rad(\mathcal{A}_{D/H})=e$. Lemma~\ref{cyclering} and Lemma~\ref{cycleorbit} imply that every basic set of $\mathcal{A}_{D/H}$ is of the form $\{x\}$ or every basic set of $\mathcal{A}_{D/H}$ is of the form $\{x,x^{-1}\}$. So $\mathcal{A}$ is regular and $D_{n-1}$ is an $\mathcal{A}$-subgroup. Since $H\leq D_{n-1}$ and $\rad(Y)=H$ for every highest basic set $Y$ of $\mathcal{A}$, we conclude that $\mathcal{A}$ is an $D_{n-1}/H$-wreath product. By the induction hypothesis, $S$-rings $\mathcal{A}_{D_{n-1}}$ and $\mathcal{A}_{D/H}$ are schurian. The basic sets of circulant $S$-ring $\mathcal{A}_{D_{n-1}/H}$ are of the form $\{x\}$ or of the form $\{x,x^{-1}\}$. Thus $\aut(\mathcal{A}_{D_{n-1}/H})_e$ has a faithful regular orbit. So $\aut(\mathcal{A}_{D_{n-1}/H})$ is $2$-isolated by Lemma~\ref{regorb}. Therefore   Corollary \ref{isolated} applied to $\widetilde{S}=D_{n-1}/H$ yields that $\mathcal{A}$ is schurian which completes the proof of the theorem.
\end{proof}

{\bf Acknowledgment.}  The author would like to thank Prof. I.~Ponomarenko for the fruithful discussions on the subject matters. His suggestions help us to improve the text significantly.

\end{document}